\documentclass[reqno, a4paper]{amsart} 
\usepackage[english]{babel}
\usepackage{amsfonts, amsmath, amsthm, amssymb,amscd,indentfirst}
\usepackage{mathtools}
\usepackage{times}
\usepackage{enumerate}
\usepackage{enumitem}
\usepackage{esint}
\usepackage{anysize} %Anysize is a simple package that can be used to set up the margins of a document.
\usepackage{mathrsfs}
\marginsize{3cm}{3cm}{2.5cm}{2.5cm}
\newtheorem{theorem}{Theorem}[section]
\newtheorem{proposition}[theorem]{Proposition}
\newtheorem{lemma}[theorem]{Lemma}
\newtheorem{definition}[theorem]{Definition}
\newtheorem{corollary}[theorem]{Corollary}
\newtheorem{ex}[theorem]{Example}
\newtheorem{remark}[theorem]{Remark}

 \numberwithin{equation}{section}

\title{Removable  singularities in the boundary for quasilinear elliptic equations}

\author{Juan A. Apaza }

\address{Juan A. Apaza, Departamento de Matem\'atica, Universidade Federal da Para\'\i ba, 58051-900, Jo\~ao Pessoa--PB, Brazil}

\email{jpablo@id.uff.br}

\author{Manassés de Souza}

\address{Manassés de Souza, Departamento de Matem\'atica, Universidade Federal da Para\'\i ba, 58051-900, Jo\~ao Pessoa--PB, Brazil}

\email{manasses.xavier@academico.ufpb.br}

\begin{document}

\maketitle

\begin{abstract}
In this work, we are interested in to study removability of a singular set in the boundary for some classes of quasilinear elliptic equations. We will approach this question in two different ways: through an asymptotic behavior at the infinity of the solutions, or through conditions in the Sobolev norm of  solutions along the direction of the singular set.
\end{abstract}

\let\thefootnote\relax
\let\thefootnote\relax

\markright{REMOVABLE  SINGULARITIES}
\markleft{REMOVABLE  SINGULARITIES}
\section{Introduction and main results}

This work addresses the issue of removable singularities for weak solutions of nonlinear elliptic equations with boundary conditions in the form:
\begin{equation} \label{1}
\left\{
\begin{aligned}
-\operatorname{div} \left[ (1+|x|) ^{-\alpha} A(\nabla u) \right]   +(1+|x|) ^{- \alpha }g(u)& =0 & & \text { in } \Omega,\\
(1+|x|)^{-\alpha} A(\nabla u) \cdot \nu + (1+|x|)^{ -\beta} b(u)   &= 0 & & \text { on }  \partial \Omega,
\end{aligned}
\right.
\end{equation}
where $\Omega = \mathbb{R}^{d} \times B^+ _R$, $B^+ _R=\{(x_{d+1}, \ldots, x_N)\in \mathbb{R}^{N-1-d}\times \mathbb{R}^+\:|\: \,\,|(x_{d+1}, \ldots, x_N)| <R\}$, $R\in (0,\infty]$, $N\geq 3$ and $d\in \{ 1, \ldots, N-2 \}$. Additionally, the singular set is denoted as $\Gamma = \{x\in \mathbb{R}^{N} \:|\: x_{d+1}=\cdots=x_{N}=0 \} $. When $d=0$, we denote $\Omega = B^+_R$ and $\Gamma = \{x\in \mathbb{R}^{N} \:|\: x=0\}$.

The singularity problems of solutions for partial differential equations are a fundamental subject in the study of the theory of partial differential equations. These problems have been extensively studied by various authors, see \textit{e.g.}, \cite{cataldo2011removability,  harvey1970removable, koskela1993removability,  Horiuchi(2001)degenerateelliptic}. For some problems related to the removability of unbounded singular sets, we refer to \cite{caffarelli2022singular, fakhi2000new,  Bidaut-Veron2014localand, juutineremovabilitylevel2005}.

Let us recall some well known results concerning isolated singular points for solutions of 
\begin{equation}\label{47}
-\sum _{j=1}^N [ a_j (x,\nabla u) ]_{x_j} + \tilde{g}(x,u)=0 \quad \text { in } \quad \Sigma,
\end{equation}
where $\Sigma$ is a bounded domain in $\mathbb{R}^N$, $a_j(x,\eta)$, $j=1, \ldots, N$, satisfy the standard conditions of ellipticity and growth with respect to $x$ and $\eta$. Furthermore, $g$ satisfies certain structural hypotheses. In the scenario where $\tilde{g}(x, u) \equiv 0$, the removability condition of an isolated singular point $x_0$ was examined by Serrin \cite{serrin1965localbehavior, serrin1965isolated}. The criterion for this situation can be expressed as follows:
$$
u(x)=o(\left|x-x_0\right|^{-\frac{N-p}{p-1}})\quad \text { and } \quad 1<p<N,
$$
see \cite{nicolosi2003precise}.  The removable singularity of a solution for the equation
\begin{equation*}
-\Delta u + \tilde{g}(x,u)=0 \quad \text { in } \quad \Sigma,
\end{equation*}
was studied by Brezis and Veron \cite{brezis1980removable}. They proved the removabilty of isolated singularities for solutions under the condition $\tilde{g}(x,u) \operatorname{sign} u \geq  |u|^q$ with $q\geq N/(N-2)$ and $N\geq 3$. For further results  in this directions, refer to  \cite{cianci2010removability}. Skrypnik \cite{skrypnik2005removability} considered the general nonlinear equations \eqref{47} and proved that if  conditions $\tilde{g}(x,u) \operatorname{sign} u \geq  |u|^q$,
$$
1<p<N\quad \text { and } \quad q \geq N(p-1)/(N-p),
$$
are fulfilled,  the singular point is removable. Also, in \cite{liskevich2008isolated}, the authors proved the removability of the isolated singular point for \eqref{47} with a weighted function $|x|^{-\alpha}$ in the subordinate part $\tilde{g}(x, u)$, provided that the conditions
$$
q \geq (p-1)(N-\alpha)/(N-p), \quad \alpha<p \quad \text { and } \quad 1<p<N,
$$
are fulfilled.

When $\sum _{j=1} ^N a_j (x,\eta) \eta _j \geq \mu w_1 (x) |\eta| ^p$, $|a_j (x,\eta)|\leq \mu ^{-1} w_1 (x)|\eta| ^{p-1}$, and $\tilde{g}(x,u) \operatorname{sign} u \geq  w_0 (x)|u|^q$, in the problem involving weighted functions $w_0$ and $w_1$, Mamedov and Harman \cite{mamedov2009removability} proved that an isolated singular point $x_0$ is removable for solutions of equation \eqref{47} if the weighted function condition
$$
w_0 (B(x_0 , \varepsilon)) \left(\frac{w_1 (B(x_0 , \varepsilon))}{\varepsilon ^p w_0 (B(x_0 , \varepsilon))} \right) ^{\frac{q}{q-p+1}}=o(1), \quad \varepsilon \rightarrow 0, 
$$
is satisfied, with $p>1$ and $q>p-1$. 

In the case $\sum _{j=1} ^N a_j (x,\eta) \eta _j \geq \mu  |\eta| ^{p(x)}$, $|a_j (x,\eta)|\leq \mu ^{-1} |\eta| ^{p(x)-1}$, and $\tilde{g}(x,u) \operatorname{sign} u \geq  |u|^{q(x)}$, where $p,q:\bar{\Sigma} \rightarrow \mathbb{R}$ are continuous functions, Fu and Shan \cite{fu2016removability} provided sufficient conditions for the removability of isolated singular points for elliptic equations in the variable exponent Sobolev space $W^{1,p(x)} (\Sigma)$. They showed that an isolated point $x_0$ is removable for solutions of equation \eqref{47} if the following conditions are met: $1 < \inf_{\bar{\Sigma}} p \leq \sup_{\bar{\Sigma}} p < N$, $\inf_{\bar{\Sigma}} (q-p-1) > 0$, and
$$
\sup _{\bar{\Sigma}}\frac{pq}{q-p+1}<N.
$$
An extension of the problem discussed by Fu and Shan is presented in \cite{apaza2022removable}. This work addresses the problem of removing singularities when the singular set is a compact smooth manifold of dimension $d_0 \in \{0, \ldots, N-2\}$ located within the boundary $\partial \Sigma$. In this paper, we have the Neumann problem
$$
\left\{ \begin{aligned}
-\sum _{j=1}^N [ a_j (x,\nabla u) ]_{x_j} + \tilde{g}(x,u) &= 0 & & \text { in }  \Sigma,\\
\sum _{j=1}^N  \frac{\partial a_j }{\partial \nu_j } (x,\nabla u)&=0 & & \text { on }  \partial \Sigma,
\end{aligned}
\right.
$$
where $1 < \inf_{\bar{\Sigma}} p \leq \sup_{\bar{\Sigma}} p < N$ and $\inf_{\bar{\Sigma}} (q-p-1) > 0$. For this case, the condition related to the dimension $d_0$ is
$$
\sup _{\bar{\Sigma}}\frac{pq}{q-p+1}<N-d_0.
$$

A different approach from  our study regarding equations with absorption or source terms
\begin{equation} \label{78}
-\Delta _p u+  |u|^{q-1} u = \mathcal{M},
\end{equation}
in a domain $\Sigma_1 \subset \mathbb{R}^N$, where $1 < p < N$, $q > p - 1$, and $\mathcal{M}$ is a Radon measure on $\Sigma_1$, is presented in \cite{Bidaut-Veron2003removsinquasi}. In that work, the removability result depends on the Bessel capacity $cap_{1, s}(F, \mathbb{R}^N)=0$, where $s > pq / (q - p + 1)$ and the singular set $F$ is a relatively closed set in $\Sigma_1$. In this context, the removability problem addresses whether a locally renormalized solution $u$ of \eqref{78} in $\Sigma_1 \backslash F$ can be extended as a locally renormalized solution of the equation in the entire $\Sigma_1$. In relation to the problem of characterizing removable sets, when $p = 2$, it was shown in \cite{barras1984singu} that they are precisely the sets $F$ with $cap_{2, q/(q-1)}(F, \mathbb{R}^N) = 0$. For the problem with nonlinear boundary conditions
$$
\left\{ \begin{aligned}
-\Delta _p u &= 0 & & \text { in }  \mathbb{R}^N_+,\\
|\nabla u|^{p-2}\frac{\partial u}{\partial \nu } &=u^q & & \text { on }  \partial \mathbb{R}^N_+.
\end{aligned}
\right.
$$
In \cite{aguirre2019harmonic}, the author provides a characterization of compact removable sets $K \subset \partial \mathbb{R}^N_+$. It was shown that they are precisely the sets $K$ with Riesz capacities $cap_{I_{p-1}, q/(q-p+1), N-1}(K) = 0$, where
$$
1<p<N \quad \text { and } \quad q>(N-1)(p-1)/(N-p).
$$

We follow \cite{mamedov2009removability, fu2016removability, apaza2022removable}, which address compact singular sets, with the Poincaré inequality playing a important role in these papers. Since the singular set $\Gamma = \mathbb{R}^d$ is unbounded, we use compactness of trace operator and continuous embedding  instead of Poincaré inequality. Regarding problem \eqref{47}, we extend the results in \cite{mamedov2009removability, fu2016removability}, where the singularity is an isolated point, see Corollary \ref{61} below.

Let   $A:   \mathbb{R} ^{N}  \rightarrow \mathbb{R} ^{N}$, $g:  \mathbb{R}   \rightarrow \mathbb{R}$  and $b: \mathbb{R}  \rightarrow \mathbb{R} $ be measurable functions with $g$, $b$  locally bounded. Initially, we assume the following conditions:

\begin{enumerate}[label=($H_{\arabic*}$)]
\item \label{43} $A (\eta) \eta \geq \mu|\eta|^{p  }$, $\left|A (  \eta)\right| \leq \mu^{-1} |\eta|^{p  -1}$,  $\forall \eta\in \mathbb{R}^N$, and $g( u) \operatorname{sign} u \geq  |u|^{q}$,    $b( u) \operatorname{sign} u \geq \mu _1  |u|^{p  - 1}$,  $\forall u \in \mathbb{R}$,
\item  \label{45}  $A( -\eta)=-A (\eta)$,  $ \forall \eta\in \mathbb{R}^N$,
\end{enumerate}
where  $p \in (1,N)$, $\alpha ,  \beta \in [0,N)$, $p   -1 \leq \beta -\alpha $, $q      - p   +1>0$,  $\mu >0$ and $\mu _1 \geq 0$.

\medskip

We begin  by considering weak solutions satisfying the  following property.

\vspace{6.5pt}

($A_1$) \  {\it Write $x^{\prime} =( x_{d+1}, \ldots, x_{N})$ and $x^{\prime \prime} = (x_{1}, \ldots,$ $ x_{d})$.  Let $u : \Omega \rightarrow \mathbb{R}$ be a measurable function. Suppose there are $\delta >d$ and $C_0>0$ such that for all $r_1, r_2 \in (0,1)$ with $r_1<r_2$ we have}
$$
u^{\pm}(x)\leq C_0 m^{\pm}_{r_1,r_2} (1+|x^{\prime \prime}|)^{-\delta}, \quad \text { \textit{a.e.} {\it in} } \{x\in \Omega \:|\: r_1< |x^\prime| <r_2\},
$$
{\it where  $m^{\pm}_{r_1,r_2}=\operatorname{ess} \operatorname{sup} \{u^{\pm}(x)\:|\:x\in \Omega, \  r_1< |x^\prime| <r_2\}$, $u^+ = \max\{u,0\} $ and $u^- = -\min\{u,0\} $.}

\vspace{6.5pt}

Our first main result is  stated as follows:

\begin{theorem}\label{23}
Suppose that \ref{43} and  \ref{45}  are satisfied $p   -1 < \beta $,  $ N-d > \max \{\frac{pq}{q-p+1}, \frac{p^2}{p-1}\}$ and $\mu _1 >0$. Under the assumption $(A_1)$, if $u\in W^{1,p} _{\operatorname{loc}}(\bar{\Omega} \backslash \Gamma) \cap L^\infty _{\operatorname{loc}} (\bar{\Omega} \backslash \Gamma )$ is a (weak) solution of \eqref{1} in $\bar{\Omega} \backslash \Gamma$. Then the singularity of $u$ at $\Gamma$ is removable.
\end{theorem}

Our second main result is motivated by the following examples, where the function $u$ blows up as we approach the singular set. 
\begin{ex} 
$ $
\begin{enumerate}
\item[(i)]  Assume that $\sigma > N-d-2$, $q>1$ and $\sigma = \frac{2}{q-1} $. Then  $\frac{2q}{q-1}> N-d$ and $u(x)=[\sigma (\sigma +2 - N+d)]^{\frac{1}{q-1}}|x^\prime|^{-\sigma}$ is a solution of
\begin{equation*} 
\left\{
\begin{aligned}
- \Delta u    +|u|^{q-1}u & =0 & & \text { in }  \mathbb{R}^{N-1}\times (0,\infty),\\
\frac{\partial u}{\partial \nu}   &= 0 & & \text { on }   \mathbb{R}^{N-1},
\end{aligned}
\right.
\end{equation*}
with no removable singularity at $\Gamma = \mathbb{R}^d$. 

\item[(ii)] Assume that $\alpha  \geq 0$, $\frac{2}{q-1} \geq \sigma > N-d-2$, $q>1$. Then  $\frac{2q}{q-1}> N-d$ and $u(x)=[\sigma (\sigma +2 - N+d)]^{\frac{1}{q-1}}|x^\prime|^{-\sigma}$  is a subsolution of
\begin{equation*} 
\left\{
\begin{aligned}
- \sum _{i=1}^N \left[ (1+|x|)^{-\alpha}  u _{x_i} \right]_ {x_i}  +(1+|x|)^{-\alpha}|u|^{q-1}u &\leq  0 & &\text { in } \mathbb{R}^{d}\times B^+ _1,\\
\frac{\partial u}{\partial \nu}  &\leq 0  & & \text { on } \partial (\mathbb{R}^{d}\times  B^+ _1).
\end{aligned}
\right.
\end{equation*}
\noindent However $[\sigma (\sigma +2 - N+d)]^{\frac{1}{q-1}}|x^\prime| ^{-\tau}\leq u(x)$ if $\tau >\frac{2}{q-1}$. Compare this example  with  Lemma \ref{19} below.
\end{enumerate}
\end{ex}

We are now working under the assumption:

\vspace{6.5pt}
 
$(A_2)$ \ {\it Let $u : \Omega \rightarrow \mathbb{R}$ be a function and $0<q_1 <((N-d)(q-p+1) -pq) /(2p)$. Suppose that $u_1 \in  W^{1,p} _{\operatorname{loc}}(\overline{B^+ _R}  \backslash \{0\}   ; (1+|x^\prime|)^{-\alpha})$ $ \cap  L^\infty _{\operatorname{loc}} (\overline{B^+ _R}  \backslash \{0\} ) $ and  $u_2  \in W^{1,p} (\mathbb{R}^{d} ; (1+|x^{\prime \prime}|)^{-\alpha}) \cap L^{q+1}(\mathbb{R}^{d} ; (1+|x^{\prime \prime}|)^{-\alpha})  \cap  \in L^{q_1} (\mathbb{R}^{d} ; (1+|x^{\prime \prime}|)^{-\alpha}) \cap L^{\infty}(\mathbb{R}^{d})$ satisfy the conditions:  $u(x)= u_1 (x^{\prime})  u_2 (x^{\prime \prime})$, $u_2 (x^{\prime \prime})>0$, for every $x\in \Omega$, and $(u_2 +|\nabla u_2|)^{p-1}\in L^{1} (\mathbb{R}^d ;  (1+|x^{\prime \prime}|)^{-\alpha})$.}

\begin{theorem}\label{77}
Suppose that  \ref{43} and \ref{45} are satisfied and  $ N-d > \frac{pq}{q-p+1}$. Under the assumption $(A_2)$, if $u\in W^{1,p} _{\operatorname{loc}}(\bar{\Omega} \backslash \Gamma ; $ $\vartheta _{\alpha}) \cap L^\infty _{\operatorname{loc}} (\bar{\Omega} \backslash \Gamma )$ is a (weak) solution of \eqref{1} in $\bar{\Omega} \backslash \Gamma$. Then the singularity of $u$ at $\Gamma$ is removable.
\end{theorem}

The hypothesis $\mu _1>0$, $p-1<\beta$ of Theorem \ref{23} differs from  $\mu _1\geq 0$, $p-1\leq \beta$ of Theorem \ref{77}, due to their distinct tools. The first result depend on the compactness of the trace operator and continuous embedding, see Propositions \ref{63}, \ref{68} and \ref{59}. Conversely, the  second result depend on Poincare inequality for bounded domains in $\mathbb{R}^{N-d}$, see \eqref{51}. Another consequences of the methods used in the proof of Theorems \ref{23} and  \ref{77} are the following results:

\begin{corollary} \label{61}
(Without the  boundary condition)  Let $\Omega = \mathbb{R}^{d} \times B _R$ and $\Gamma = \{x\in \mathbb{R} ^N \:|\: x_{d+1} = \cdots =x_N=0\}$, where $B_R=\{(x_{d+1}, $ $\ldots, x_N)\in \mathbb{R}^{N-d}\:|\: |(x_{d+1}, \ldots, x_N)| <R\}$, $R\in (0,\infty]$,  $N\geq 3$ and $d\in \{ 1, \ldots, N-2 \}$. If $d=0$, we write $\Omega = B_R$ and $\Gamma = \{x\in \mathbb{R} ^N \:|\: x=0\}$. Let $a : \mathbb{R} \rightarrow \mathbb{R} $ be a measurable functions locally bounded, satisfying  $a(u) \operatorname{sign} u\geq \mu _1 |u|^{p-1}$, $\forall u \in \mathbb{R}$. 
\begin{enumerate}
\item[(a)]  Assume that \ref{43} and \ref{45} are satisfied, $ N-d > \max \{\frac{pq}{q-p+1}, \frac{p^2}{p-1}\}$ and $\mu _1>0$. Under the assumption $(A_1)$, suppose that  $u \in W_{\operatorname{loc}  }^{1, p  }(\bar{\Omega}   \backslash \Gamma) \cap L_{\operatorname{loc}  }^{\infty}(\bar{\Omega }  \backslash \Gamma)$ is  a (weak) solution  in $\bar{\Omega}   \backslash \Gamma$ of the equation
\begin{equation}\label{271}
- \operatorname{div}[A( \nabla u )]  +  a(u) + g(u) = 0 \quad \text { in } \Omega .
\end{equation} 
Then the singularity of $u$ at $\Gamma$ is removable.

\item[(b)] Assume that \ref{43} and \ref{45} are satisfied, $N-d >\frac{pq}{q-p+1}$ and $\mu _1 \geq 0$. Under the assumption $(A_2)$, suppose that  $u \in W_{\operatorname{loc}  }^{1, p  }(\bar{\Omega}   \backslash \Gamma ; \vartheta _{\alpha}) \cap L_{\operatorname{loc}  }^{\infty}(\bar{\Omega}   \backslash \Gamma)$ is  a (weak) solution  in $\bar{\Omega}   \backslash \Gamma$ of the equation
\begin{equation}\label{71}
- \operatorname{div}[ (1+|x|)^{-\alpha} A(\nabla u)]  +  (1+|x|)^{-\alpha}a(u) + (1+|x|)^{-\alpha} b(u) = 0\quad \text { in } \Omega.
\end{equation} 
Then the singularity of $u$ at $\Gamma$ is removable.
\end{enumerate}
\end{corollary}

\begin{corollary}\label{24}
Suppose that \ref{43} and  \ref{45}  are satisfied, $p   -1 < \beta$, $N-d >\frac{pq}{q-p+1}$ and $\mu_1 >0$. Assume that $u\in W^{1,p} _{\operatorname{loc}}(\bar{\Omega} \backslash \Gamma) \cap L^\infty _{\operatorname{loc}} (\bar{\Omega} \backslash \Gamma )$ is a (weak) solution of \eqref{1} in $\bar{\Omega} \backslash \Gamma$, such that $\{x^{\prime} \:|\: x\in \operatorname{supp} u\}$ is a bounded set, $|\operatorname{supp} u \cap \{ |x^\prime| \leq r_1\} | \leq C r_1 ^{N -d}$ and
\begin{equation*}\label{25} 
 \int _{\operatorname{supp} u \cap \{r_1\leq |x^\prime|\leq r_2\}} |x^\prime| ^{-\frac{pq}{q-p+1}}\textnormal{d}x \leq C(r_2^{N-d-\frac{pq}{q-p+1}} - r_1^{N-d-\frac{pq}{q-p+1}})\quad \text { if } 0<r_1 <r_2< 1,
\end{equation*}
where   $C>0$ is a suitable constant. Then the singularity of $u$ at $\Gamma$ is removable.
\end{corollary}

Observe that  $(A_1)$ and  $(A_2)$ are not assumed in the  corollary above.

\begin{remark}
Applying the method used in the results above, the same results are valid for  weak solutions $u\in W^{1,p} (\Omega ; (1+|x^{\prime \prime}|)^{-\alpha})$ of nonlinear elliptic equations in the form:
\begin{equation*} 
\left\{
\begin{aligned}
-\operatorname{div} \left[ (1+|x^{\prime \prime}|) ^{-\alpha} A(\nabla u) \right]  +(1+ |x^{\prime \prime}|) ^{- \alpha }g(u)& =0 & & \text { in } \Omega,\\
 (1+|x^{\prime \prime}|) ^{-\alpha} A(\nabla u) \cdot \nu   + + (1+|x|)^{ -\beta }b(u)  &= 0 & & \text { on }  \partial \Omega,
\end{aligned}
\right.
\end{equation*}
and 
\begin{equation*} 
-\operatorname{div} \left[ (1+|x^{\prime \prime}|) ^{-\alpha} A(\nabla u) \right]  + (1+|x^{\prime \prime}|)^{ -\alpha }a(u)  +(1+ |x^{\prime \prime}|) ^{- \alpha }g(u) =0 \quad \text { in } \Omega.
\end{equation*}
Observe that $W^{1,p} (\Omega ; (1+|x^{\prime \prime}|)^{-\alpha}) \subset W^{1,p} (\Omega ; (1+|x|)^{-\alpha})$, then from $ W^{1,p} (\Omega ; (1+|x|)^{-\alpha}) \rightarrow L^{p} (\partial \Omega ; (1+|x|)^{-\beta}) $ we have the trace $u|_{\partial \Omega}$ has definite meaning for all $u\in W^{1,p} (\Omega ; (1+|x^{\prime \prime}|)^{-\alpha})$. Also, it is important to note that near the singular set, we have $(1+|x^{\prime \prime}|)^{-\alpha}\leq C (1+|x|)^{-\alpha}$ in $\{|x^{\prime}|<r_0\}$, $r_0>0$, for some positive constant $C$.
\end{remark}

This paper is organized as follows. In Section \ref{31}  we gather preliminary definitions and results,  which is used several times in the paper, like weak solutions, embedding theorem, and equivalent norms in Sobolev spaces. In Section \ref{32} we prove  Proposition \ref{66} and \ref{266}, crucial to our main results, these propositions deal with the behavior of subsolutions near the singular set. In Section \ref{33} the proof of Theorems \ref{23} and  \ref{77} is given, here we prove the boundedness of weak solutions and provide an estimate of the $L^p$-norm of the gradient.

\section{Preliminaries} \label{31}

In order to simplify the representation we will denote $\vartheta _{\alpha} (x) = (1+|x|)^{-\alpha}$, $ \textnormal{d}\vartheta _{\alpha} = \vartheta _{\alpha} \textnormal{d}x $  and $\vartheta _{\alpha} (F) = \int _F \vartheta _{\alpha} \textnormal{d}x$, where $F$ is a measurable set. Let $U\subset \mathbb{R}^{N}$ be an open set and  $p >1$, define
$$
L^{p}(U ; \vartheta _{\alpha})=\left\{u \:|\: u : \Omega \rightarrow \mathbb{R} \text { is  measurable function   and } \int_{U } |u|^{p} \textnormal{d}\vartheta _{\alpha} <\infty\right\},
$$
and
$$
W^{1, p  }(U  ; \vartheta _{\alpha} )=\left\{u \in L^{p  }\left(U   ; \vartheta _{\alpha}\right)\:|\:  |\nabla u| \in L^{p  }(U  ; \vartheta _{\alpha} )\right\},
$$
equipped with the norms
$$
\|u\|_{p  , U  , \alpha }=\left( \int_{U} |u|^{p} \textnormal{d}\vartheta _{\alpha}  \right)^{\frac{1}{p}} \quad \text { and } \quad  \|u\|_{1, p  , U  , \alpha }=\|u\|_{p  , U  , \alpha} + \|\nabla u\|_{p  , U , \alpha}, 
$$
respectively.  For details on these spaces see \cite{horiuchi1989imbedding, gurka1991continuous, pfluger1998compact, liu2008compact}.

We work with the following spaces
\begin{align*}
L_{\operatorname{loc} }^{p  }(\bar{\Omega} \backslash  \Gamma  ; \vartheta _{\alpha})
=\{u: \Omega \rightarrow \mathbb{R} \:|\: u \in L^{p  }(U  ; \vartheta _{\alpha}) \text {  for all open subset } U \subset \Omega   \text { so that }  \bar{U} \cap \Gamma=\emptyset \}
\end{align*}
and 
\begin{align*}
W_{\operatorname{loc} }^{1, p  }(\bar{\Omega} \backslash  \Gamma ; \vartheta _{\alpha} )
=\{u: \Omega \rightarrow \mathbb{R} \:|\: u \in W^{1, p  }(U ; \vartheta _{\alpha} ) \text { for all open subset } U \subset \Omega  \text { so that }   \bar{U} \cap \Gamma=\emptyset  \}.
\end{align*}

Similarly we define $L_{\operatorname{loc} }^{\infty}(\bar{\Omega} \backslash \Gamma)$. 

\begin{definition}
 We will say that $u \in W_{\operatorname{loc}}^{1, p}(\bar{\Omega} \backslash \Gamma ; \vartheta _{\alpha} ) \cap L_{\operatorname{loc}}^{\infty}(\bar{\Omega} \backslash \Gamma)$ is a (weak) solution of  \eqref{1} in $\bar{\Omega} \backslash \Gamma$ if
\begin{equation}\label{62}
\int_{\Omega}\left(\langle A(  \nabla u), \nabla \varphi\rangle  +  g( u) \varphi \right) \textnormal{d}\vartheta _{\alpha } + \int_{\partial \Omega} b( u)   \varphi \textnormal{d}\vartheta _{\beta}=0,
\end{equation}
for all $\varphi \in W_{\operatorname{loc}}^{1, p }(\bar{\Omega} \backslash \Gamma ; \vartheta _{\alpha}  ) \cap L_{\operatorname{loc}}^{\infty}(\bar{\Omega} \backslash \Gamma)$, such that $\operatorname{supp} \varphi \subset \bar{\Omega} \backslash \Gamma$.
\end{definition}

Let us observe that the trace of $u \in W_{\operatorname{loc} }^{1, p }(\bar{\Omega}   \backslash \Gamma ; \vartheta _{\alpha}) \cap L_{\operatorname{loc}  }^{\infty}(\bar{\Omega}   \backslash \Gamma)$ may not be defined on $\partial \Omega  $, however it is defined on $\{x \in \partial \Omega   \:|\: |x^\prime|  >r\}$ and is essentially bounded for small enough values $r>0$.
\begin{definition}
 We will say that the solution $u$ of  \eqref{1} in $\bar{\Omega} \backslash \Gamma$ has a removable singularity at $\Gamma$ if $u \in W^{1, p  }(\Omega ; \vartheta _{\alpha} ) \cap L^{\infty}(\Omega)$ and the equality \eqref{62} is fulfilled for all $\varphi \in W^{1, p  }(\Omega   ;  \vartheta _{\alpha} ) \cap L^{\infty}(\Omega)$.
 \end{definition}

We finish this section by collecting some facts that we will utilize  throughout this paper.
\begin{proposition} \label{63} (see \cite{pfluger1998compact, liu2008compact})
Suppose  that $p \in (1,N)$, $q\geq p$,  $ \alpha  , \beta \in [0, N)$  and $0 \leq  \frac{N-1}{q} - \frac{N}{p} +1 \leq  \frac{\beta}{q} - \frac{\alpha }{p}$.  Then  the trace operator $ W^{1,p} (\Omega ; \vartheta _{\alpha}) \rightarrow L^q (\partial \Omega  ; \vartheta _{\beta})$ is continuous. If in addition $0 \leq  \frac{N-1}{q} - \frac{N}{p} +1 <   \frac{\beta}{q} - \frac{\alpha }{p}$, then  the trace operator is compact.
\end{proposition}

\begin{proposition} \label{68}
Assume that $p \in (1,N)$, $q\geq p$,  $ \alpha  , \beta \in [0, N)$  and $0 \leq  \frac{N-1}{q} - \frac{N}{p} +1 < \frac{\beta}{q} - \frac{\alpha }{p}$.  For any $u\in W^{1,p} (\Omega ; \vartheta _{\alpha})$, let
$$
\|u\|_{\partial}=  \|u\|_{ p, \partial \Omega , \beta} + \|\nabla u\|_{p, \Omega  , \alpha }.
$$
Then $\|u\|_{\partial}$ is a norm on $W^{1, p }(\Omega ; \vartheta _{\alpha})$ which is equivalent to
$$
\|u\|_{1, p, \Omega  , \alpha }=\|u\|_{p, \Omega  , \alpha } + \|\nabla u\|_{p, \Omega  , \alpha} .
$$
\end{proposition}

\begin{proof}
Follows of steps of the proof of \cite[Theorem 2.1]{deng2009positive}. Indeed, by Proposition \ref{63}, there exists $C>0$ such that $\|u\|_{\partial} \leq C\|u\|_{1,p,\Omega,\alpha}$, for every $u\in W^{1,p}(\Omega,\vartheta _{\alpha})$. Next we will prove 
$$
\|u\|_{p,\Omega,\alpha} \leq C_1\|u\|_{\partial},
$$
for a suitable constant $C_1>0$. Arguing by contradiction, there is a sequence $(u_n) \subset  W^{1, p}(\Omega , \vartheta _{\alpha})$ such that
$$
n\|u_n\|_{\partial}<\|u_n\|_{p,\Omega,\alpha} .
$$
We can assume that $\|u_n\|_{p,\Omega,\alpha}=1$ for all $n$. Then $\left\|u_n\right\|_{\partial} \rightarrow 0$, which implies
\begin{equation*}\label{49}
\|\nabla u_n\|_{p,\Omega,\alpha} \rightarrow 0 \quad \text { and } \quad \|u_n\|_{p,\partial \Omega,\beta} \rightarrow 0.
\end{equation*}
Consequently, $(\|u_n\|_{1,p,\Omega,\alpha})$ is bounded. By the reflexivity of $W^{1, p}(\Omega , \vartheta _{\alpha})$, there exists a subsequence, still denoted by $(u_n)$,  such that
\begin{equation}\label{35}
u_n \rightharpoonup u_0 \quad\text{in}\quad W^{1, p}(\Omega,\vartheta _{\alpha}) .
\end{equation}
Furthermore, due to the compact embedding theorem, we can assume that
\begin{equation}\label{41}
u_n \rightarrow u_0, \quad\text{in}\quad L^{p}(\Omega _m , \vartheta _{\alpha}) \text { for every } m,
\end{equation}
where $\Omega _m = \Omega \cap \{|x| < m\} $, $m \in \mathbb{N}$. From \eqref{35} and \eqref{41}, we have $ \nabla u_n \rightharpoonup \nabla u_0 $ in $(L^{p}(\Omega _m , \vartheta _{\alpha}))^N$. Thus,
$$
\|\nabla u_0\|_{p , \Omega _m , \alpha} \leq \liminf _{n \rightarrow \infty}\|\nabla u_n\|_{p,\Omega _m,\alpha} =0.
$$
Therefore, $u_0\equiv \operatorname{constant}$. Since $W^{1, p}(\Omega , \vartheta _{\alpha}) \rightarrow L^{p}(\partial \Omega , \vartheta _{\beta})$ is compact, we have
$$
\|u_0\|_{p,\partial \Omega,\beta} \leq \liminf _{n \rightarrow \infty}\|u_n\|_{p,\partial \Omega,\beta} =0,
$$
then, $u_0 \equiv 0$. Which is contradiction, because $\|u_n\|_{p,\Omega,\alpha}=1$, for every  $n$. This concludes the proof. \end{proof}

\begin{proposition} \label{59}  (see  \cite[Theorem 3]{horiuchi1989imbedding}) Suppose that $p \in (1,N)$, $q\geq p$,  $0\leq \frac{N}{p} - \frac{N}{q} \leq 1+ \frac{\alpha}{p} - \frac{\alpha _0}{q} < \frac{N}{p}$ and  $\frac{N}{q} \geq \frac{\alpha _0}{q} \geq \frac{\alpha}{p} \geq 0$. Then there is a continuous embedding $W^{1,p} (\Omega ; \vartheta _{\alpha}) \rightarrow L^{q} (\Omega ; \vartheta _{\alpha _0})$.
\end{proposition}

\begin{lemma}\label{64} (see \cite[ p. 1004]{10.1215/S0012-7094-84-05145-7}) Let $0<\theta<1$, $\sigma>0$, $\xi(h)$ be a nonnegative function on the interval $[1 / 2,1]$, and let
$$
\xi(k) \leq C_0(h-k)^{-\sigma}(\xi(h))^\theta, \quad 1 / 2 \leq k<h \leq 1,
$$
for some positive constant $C_0$. Then, there exists $C_1(\sigma, \theta)>0$ such that
$$
\xi(1 / 2) \leq C_1 C_0^{\frac{1}{1-\theta}} .
$$
\end{lemma}

\section{The behavior of solutions near the singular set} \label{32}

In order to achieve our main results, we require the following analytic results regarding the behavior near the singular set of solutions to \eqref{1}.

\begin{proposition} \label{66}
Assume that \ref{43} and \ref{45} are  satisfied and $1<p<N-d$. Under the assumption $(A_2)$, suppose that $u \in W_{\operatorname{loc}  }^{1, p  }(\bar{\Omega}   \backslash \Gamma ; $ $\vartheta _{\alpha})\cap L_{\operatorname{loc}  }^{\infty}(\bar{\Omega}   \backslash \Gamma)$ is a solution of \eqref{1} in $\bar{\Omega}   \backslash \Gamma$. Then, if $\varepsilon _0 >0$,
\begin{equation}\label{75}
|u_1(x^\prime)| \leq C |x^\prime|  ^{-\tau}, \quad \text { \textit{a.e.} in } \left\{x^\prime \in B^+ _R   \:|\: 0<|x^\prime|  <r_0\right\}  ,
\end{equation}
where $C=C(N, \mu, \alpha , \beta , d , \varepsilon _0 , p  , q , \|u_2 \|_{q+1,\mathbb{R}^d,\alpha} , \|u_2\| _{1,p, \mathbb{R}^d,\alpha}   , \mathcal{U})>0$ and $\tau=\tau(N, \mu, \alpha , \beta ,  d ,  \varepsilon _0 , p  , $ $ q) \in (\frac{p}{q-p+1} ,  \frac{p}{q-p+1} + \varepsilon _0 )$.
\end{proposition}

\begin{proposition} \label{266}
Assume that  \ref{43} and  \ref{45}  are satisfied, $p   -1 < \beta $ and $\mu_1>0$. Under the assumption $(A_1)$, suppose that $u \in W_{\operatorname{loc}  }^{1, p  }(\bar{\Omega}   \backslash \Gamma )\cap L_{\operatorname{loc}  }^{\infty}(\bar{\Omega}   \backslash \Gamma)$ is a solution of \eqref{1} in $\bar{\Omega}   \backslash \Gamma$. Then, if $\varepsilon _0 >0$,
\begin{equation}\label{275}
|u(x)| \leq C |x|  ^{-\tau}, \quad \text { \textit{a.e.} in } \left\{x \in \Omega   \mid 0<|x^\prime|  <r_0\right\}  ,
\end{equation}
where $C=C(N, \mu,  \beta , d , \varepsilon _0 , p  , q ,   C_0 ,\mathcal{U} )>0$ and $\tau=\tau(N, \mu, \beta , d ,  \varepsilon _0 , p  ,  $ $ q) \in (\frac{p}{q-p+1} , $ $ \frac{p}{q-p+1} + \varepsilon _0 )$.
\end{proposition}

Let $r_0 \in (0,1)$ be small enough,  write  $\mathcal{U}=\{x\in \Omega \:|\: |x^\prime|   < 2 r_0 \}$. For $\ell\in \left(0, r_0\right)$ and $r\in (0,\ell)$, we define $V_{\ell,r}= \{x\in \Omega  \:|\: \left| |x^\prime|   - \ell \right|< r \}$ and $V_{\ell,r} ^{\prime }= \{x^\prime \in B^+ _R \:|\: \left| |x^\prime|   - \ell \right|< r \}$.

\subsection{Proof of Proposition \ref{66}} 

The proof  follows from the next lemma.

\begin{lemma} \label{19}
Assume that   \ref{43} is  satisfied and $1<p<N-d$. Under the assumption $(A_2)$, suppose that   $u\in  W^{1, p  } _{\operatorname{loc}} (\bar{\Omega} \backslash \Gamma ; \vartheta _{\alpha}  )$ $\cap L^{\infty} _{\operatorname{loc}} (\bar{\Omega}  \backslash \Gamma )$ satisfies 
\begin{equation} \label{22}
\int _{\Omega } \left(\left\langle A(\nabla u) , \nabla \varphi \right\rangle  +  g (u) \varphi \right)  \textnormal{d}\vartheta _{\alpha}  + \int _{\partial \Omega} b( u)   \varphi \textnormal{d}\vartheta _{\beta} \leq 0,
\end{equation}
for all $\varphi \in W_{\operatorname{loc}}^{1, p }(\bar{\Omega}  \backslash \Gamma  ;\vartheta _{\alpha}) \cap L_{\operatorname{loc}}^{\infty}(\bar{\Omega}  \backslash \Gamma )$, $\varphi \geq 0$, with   $\operatorname{supp} \varphi  \subset \bar{\Omega } \backslash \Gamma $. Then, if $\varepsilon _0 >0 $ and  $0<r<\ell < r_0$ we have the estimate
\begin{equation} \label{42}
\max \{u_1 ,0\}   \leq C r^{-\tau}, \quad \text { \textit{a.e.} in } V_{\ell, r/2} ^\prime ,
\end{equation}
where $C=C(N,  \mu  , \alpha , \beta , d , \varepsilon _0 , p , q  ,  \|u_2 \|_{q+1,\mathbb{R}^d,\alpha} , \|u_2\| _{1,p, \mathbb{R}^d,\alpha} , \mathcal{U} )>0$ and  $\tau=\tau (N,  \mu  , \alpha , \beta ,  $ $ d ,  \varepsilon _0 , p , q   )\in ( \frac{p}{q- p   +1} ,  \frac{p}{q- p   +1} +\varepsilon _0)$. 
\end{lemma} 
\begin{proof} 
First note that, if  $v \in W^{1, p}(B^+_{2r_0})$ and $v=0$ on $\{x^\prime \in \overline{B^+_{2r_0}} \:|\: |x^{\prime}|=2 r_0\}$, we have
\begin{equation}\label{51}
\left(\int_{B^+_{2r_0}}|v|^q \textnormal{d}x^\prime \right)^{\frac{1}{q}} \leq C\left(\int_{B^+_{2r_0}}|\nabla v|^p \textnormal{d}x^\prime \right)^{\frac{1}{p}}
\end{equation}
for each $q \in [p, \frac{(N-d) p}{N-d-p})$, where $C=C( N-d , p , q , B^+_{2r_0})>0$. The proof of \eqref{51} proceeds by contradiction, by using that $W^{1, p}(B^+_{2r_0})$ is compactly embedded in $L^q (B^+_{2r_0})$. 

Next we divide the proof of the lemma into three steps: 

$\it{Steep \ 1.}$  \ In the sequel, we  use $C_i$, $i=1, 2,\ldots$, to denote  suitable positive constants. Let $u_1 = C_\ast w_1 $, where  $C_\ast>0$ is a  number that will be determined below. We assume that $|\{x^\prime \in V_{\ell,r/2} ^\prime  \:|\: w_1(x^\prime)>0\} |\neq 0$, otherwise \eqref{42} is immediate.  Set $\Omega _+   =\{x \in V_{\ell,r}  \:|\: w_1(x^\prime)>0\}$. Take
$$
m_{t}=\operatorname{ess} \operatorname{sup} \left\{  w_1(x^\prime)\:|\: x \in V_{\ell,t r} \cap \Omega _+   \right\}, \quad 1/2 \leq t \leq 1 .
$$
 Let $1/2 \leq s<t \leq 1$.  Define the functions $z : \Omega  \rightarrow \mathbb{R}$, $z_k  :\Omega  \rightarrow \mathbb{R}$, by
\begin{align*}
z (x)&=   w_1(x^\prime) - m_{t}   \xi \left(\left|  |x^\prime|   - \ell \right|\right)  ,\\
z_k  (x)&=\left\{ \begin{aligned}
& \max \left\{ w(x)-   \left ( m _{t}   \xi \left(\left|  |x^\prime|  - \ell \right|\right) + k\right) u_2(x^{\prime \prime}), 0\right\}  & & \text { if } x\in  V_{\ell , tr},\\
& 0 & & \text { if } x\in \Omega  \backslash V_{\ell , tr},
\end{aligned}
\right. 
\end{align*}
where  $w=w_1 u_2$,  $0 \leq k\leq \operatorname{ess} \operatorname{sup} _{\Omega _+ } z$, and $\xi : \mathbb{R} \rightarrow \mathbb{R}$ is a smooth function satisfying: $\xi = 0$ on $(-\infty , s r]$, $\xi  = 1$ on $ \left[ \frac{s+t}{2} r , \infty \right)$,
$$
0 \leq \xi \leq 1 \quad\text{and}\quad \left| \xi^{\prime}\right| \leq \frac{C_1}{r(t-s)} \quad\text{on}\quad    \mathbb{R}.
$$
Observe that $ z_{k} \in W^{1, p }(\bar{\Omega} \backslash \Gamma ;  \vartheta _{\alpha} )$ and  $ \operatorname{supp} z_k\subset \bar{\Omega} \cap\{ \left| |x^\prime|   - \ell \right| \leq tr \}\subset \bar{\Omega}  \backslash \Gamma $.   For simplicity we write $\xi\left( \left|  |x^\prime|  - \ell \right| \right) = \xi (x)$ and $\xi ^\prime \left( \left|  |x^\prime|  - \ell \right| \right) = \xi ^\prime (x)$. 

Take $k \in [0, \mathcal{K} )$, where $\mathcal{K} = \sup \{k\in [0 , \operatorname{ess} \operatorname{sup} _{\Omega   ^\prime} z] \:|\:  |\{ x\in V_{\ell , tr} \:|\: z_k (x) >0\} |\neq 0\}$. Observe that $\mathcal{K} \geq m_s \geq m_{1/2} $. Substituting $\varphi =  z_k$ into \eqref{22}, we obtain
\begin{equation*} 
\int _{\Omega }  \left( \left\langle A (\nabla u) , \nabla z_k  \right\rangle  +  g( u) z_k \right) \textnormal{d} \vartheta _{\alpha} + \int _{\partial \Omega} b(u)    z_k \textnormal{d}\vartheta _{\beta} \leq 0. 
\end{equation*}
Denote $\Omega _k  =\{x \in V_{\ell, t r}\:|\: z_k (x)>0\}$. Then,
\begin{align*}
  &\int_{ \Omega_k  }\left( \left\langle A \left( C_\ast \nabla w\right)  ,     \nabla w     - m _{t} u_2 \xi  ^{\prime}   \nabla ||x^\prime| -\ell |    
 - \left ( m _{t}   \xi + k\right) \nabla u_2 \right\rangle \right)\textnormal{d}\vartheta _{\alpha} \\ 
  &+ \int _{\partial \Omega   _k \cap \partial \Omega }  b( C_\ast w)  z_k \textnormal{d}\vartheta _{\beta}  
   +\int_{ \Omega  _k  }  g( C_\ast w) z_k \textnormal{d} \vartheta _{\alpha}   \leq 0.
\end{align*}
By \ref{43}, we have
\begin{equation*} \label{36}
\begin{aligned}
&\int_{\Omega _k }  \mu   C_{\ast} ^{p - 1}|\nabla w|^{p}   \textnormal{d}\vartheta _{\alpha}  +  \int_{\partial \Omega  _k \cap \partial \Omega }  \mu _1 |C_\ast w|^{p-1} z_k   \textnormal{d}\vartheta _{\beta}   \\
&+ \int_{\Omega _k }   |C_\ast w|^{q} z_k   \textnormal{d} \vartheta _{\alpha }   \leq \int _{\Omega _k}\mu ^{-1} \left[ \frac{ C_1 m_t}{r(t-s)} u_2 +  (m_t +k) |\nabla u_2| \right]  |C_\ast \nabla w|^{p-1} \textnormal{d} \vartheta _{\alpha }.
\end{aligned}
\end{equation*}
From $m_t u_2 \geq w \geq k u_2$ and $ w \geq z_k$ on $\Omega _k$, we have 
\begin{equation}\label{37}
\begin{aligned} 
&\int_{\Omega  _k } \mu    C_{\ast} ^{p - 1}|\nabla w|^{p}    \textnormal{d}\vartheta _{\alpha}   +    \int_{\partial \Omega  _k \cap \partial \Omega } \mu _1  C_\ast ^{p-1}  z_k ^p   \textnormal{d}\vartheta _{\beta}  \\
&+ \int_{\Omega  _k}     C_\ast^{q} k^{q} z_k  u_2 ^q \textnormal{d} \vartheta _{\alpha} \leq  C_2   \int_{\Omega  _k } \mu ^{-1} \frac{m_t C_\ast ^{p-1}}{r(t-s)}  (u_2 + |\nabla u_2|) | \nabla w|^{p-1}  \textnormal{d} \vartheta _{\alpha} .
\end{aligned}
\end{equation}
On the other hand, using Young's inequality,
\begin{equation} \label{38}
\begin{aligned}
&  \int_{\Omega  _k }  \frac{  m _t  C_\ast ^{p-1}}{r(t-s)}  (u_2 + |\nabla u_2|) |\nabla w|^{p-1} \textnormal{d} \vartheta _{\alpha}  \\
&\leq  \int_{\Omega  _k }   C_\ast ^{p-1} C_3 \left[\frac{m _t}{r(t-s)}\right]^{p} (u_2 ^p+ |\nabla u_2|^p)\textnormal{d} \vartheta _{\alpha } + \varepsilon _1 \int _{\Omega _k}C_\ast ^{p-1} |\nabla w|^{p}   \textnormal{d}\vartheta _{\alpha} .
\end{aligned}
\end{equation}
Take  $\varepsilon _1 = \frac{\mu ^2}{2C_2}$. Then, by \eqref{37} and \eqref{38},
\begin{align*}
&\frac{\mu }{2}\int_{\Omega _k} C_{\ast} ^{p- 1}  |\nabla w|^{p}    \textnormal{d}\vartheta _{\alpha}  +  \mu _1 \int_{\partial \Omega  _k \cap \partial \Omega} C_\ast ^{p-1} z_k ^p \textnormal{d}\vartheta _{\beta}  \\
&+ \int_{\Omega _k } C_\ast ^{q} k^{q}z_k  u_2 ^q\textnormal{d} \vartheta _{\alpha }  \leq    C_4  C_\ast ^{p-1} \left[\frac{m_t}{r (t-s)} \right]^{p  }  \int_{\Omega  _k }   \left(u_2 ^p+ |\nabla u_2|^p\right) \textnormal{d}\vartheta _{\alpha}.
\end{align*}
Note that $\nabla z_{k}=\nabla w-m_{t} u_2 \xi ^{\prime}  \nabla \left(| |x^\prime|   - \ell |\right) - (m_t\xi +k) \nabla u_2  $ in $\Omega  _k $, therefore
\begin{equation}\label{11}
\begin{aligned}
&\frac{\mu  C_{\ast} ^{p - 1} }{2}\int_{\Omega_k } \left( \frac{1}{2^{2(p-1)}}\left|\nabla z_{k}\right|^{p}-( \left|\xi ^\prime\right|^{p} u_2 ^p + 2^p |\nabla u_2|^p) m_{t}^{p} \right)   \textnormal{d}\vartheta _{\alpha}  
+ \mu  _1 C_\ast ^{p-1} \int_{\partial \Omega  _k \cap \partial \Omega}  z_k ^p \textnormal{d}\vartheta _{\beta}   \\
 &+ C_\ast ^{q} k^q \int_{\Omega  _k } z_k u_2 ^q\textnormal{d} \vartheta _{\alpha}
 \leq    C_4 C_\ast ^{p -1 } \left[\frac{m_t}{r (t-s)} \right]^{p} \int_{\Omega_k }  \left( u_2 ^p+ |\nabla u_2|^p \right)\textnormal{d}\vartheta _{\alpha}.
\end{aligned}
\end{equation}
Then, 
\begin{equation*} 
\begin{aligned}
&\frac{\mu  C_{\ast} ^{p - 1} }{2^{2p -1 }}\int _{\Omega  _k }    |\nabla z_{k}|^{p  } \textnormal{d}\vartheta _{\alpha}  +  \mu _1 C_{\ast} ^{p - 1}  \int _{\partial \Omega _k \cap \partial \Omega}  z_k ^{p  }  \textnormal{d}\vartheta _{\beta}  \\
&+  C_\ast ^{q} k^{q }\int_{\Omega  _k } z_k u_2 ^q \textnormal{d} \vartheta _{\alpha  } \leq    C_5  C_\ast ^{p -1 } \left[\frac{m_t}{r (t-s)} \right]^{p  } |\Omega _k ^\prime | \|u_2\| ^p _{1,p,\mathbb{R}^d,\alpha},
\end{aligned}
\end{equation*}
where $\Omega _k ^\prime=\{x^{\prime} \in V^\prime _{\ell, tr}\:|\: w_1(x^\prime)-    m _{t}   \xi \left(\left|  |x^\prime|  - \ell \right|\right) - k>0\}$. Hence,
\begin{equation} \label{14}
\begin{aligned}
&\mu  C_6 C_{\ast} ^{p - 1} \int _{\Omega  _k ^{\prime} } |\nabla \hat{z}_{k}|^{p} \textnormal{d}x^\prime \|u\|_{p,\mathbb{R}^d,\alpha} ^p  + C_\ast ^{q} k^{q}\int_{\Omega _k } z_k u_2 ^q \textnormal{d} \vartheta _{\alpha}\\
& \leq \frac{\mu  C_{\ast} ^{p - 1} }{2^{2p -1 }}\int _{\Omega  _k }    |u_2\nabla \hat{z}_{k} + \hat{z}_k \nabla u_2|^{p} \textnormal{d}\vartheta _{\alpha}  +  C_\ast ^{q} k^{q }\int_{\Omega  _k } z_k u_2 ^q \textnormal{d} \vartheta _{\alpha  }\\
&\leq C_5  C_\ast ^{p -1 } \left[\frac{m_t}{r (t-s)} \right]^{p} |\Omega _k ^\prime | \|u_2\| ^p _{1,p,\mathbb{R}^d,\alpha},
\end{aligned}
\end{equation}
where $u_2(x^{\prime \prime} )\hat{z}_k (x^\prime)= z_k (x)$.

On the other hand. We assume that  $\gamma _1  , \gamma _2 >0$ and $\varepsilon \in (0,1)$  satisfies $p< \gamma _2 < \gamma _1  <\frac{(N-d)p}{N-d -p}$, $\frac{p  }{q-p+1} + \varepsilon_0 > \frac{ p}{(q   - p   +1)(1-\varepsilon)} $ and $1 > (N -d)( 1- \frac{p}{\gamma _1}  )$. Then, by  \eqref{51},
\begin{equation}\label{39} 
\begin{aligned}
&\left(\int_{\Omega  _k ^\prime} \hat{z}_{k}^{\gamma _2  } \textnormal{d}x^\prime \right)^{\frac{1}{\gamma _2  }}\leq \left(\int _{\Omega  _k ^\prime} \hat{z}_k ^{\gamma _1} \textnormal{d}x^\prime \right)^{\frac{1}{\gamma _1}} |V_{\ell , tr} ^\prime|^{\frac{\gamma _1 - \gamma _2}{\gamma _1 \gamma _2}} \\  
&\leq  C_7 r^{\frac{(N-d)(\gamma _1   - \gamma _2  )}{\gamma _1  \gamma _2  }} \left(\int _{\Omega _k ^\prime}\left|\nabla \hat{z}_{k}\right|^{p} \textnormal{d}x^\prime   \right)^{\frac{1}{p }}.
\end{aligned}
\end{equation}
Using Hölder's inequality, we get
\begin{align}
&\int_{\Omega _k ^\prime} \hat{z}_{k} \textnormal{d}x^\prime  \leq \left(\int_{\Omega  _k ^\prime} \hat{z}_{k}^{ \gamma _2} \textnormal{d}x^\prime \right)^{\frac{1}{\gamma _2  }}|\Omega _k ^\prime |^{\frac{\gamma _2   -1}{ \gamma _2}} \label{40}.
\end{align} 
From \eqref{14} - \eqref{40}, we have
\begin{equation} \label{15}
\begin{aligned}
& \mu C_6 C_\ast  ^{p   - 1} \left[  C_7 ^{-1} |\Omega  _k ^\prime | ^{-\frac{\gamma _2   -  1}{\gamma _2  }}  r^{-\frac{(N-d)(\gamma _1  - \gamma _2  )}{\gamma _1  \gamma _2  }}   \int_{\Omega _k ^\prime} \hat{z}_k \textnormal{d}x^\prime  \right] ^{p} \|u_2\|_{p,\mathbb{R}^d,\alpha}^p +  C_\ast ^{q} k^q \int _{\Omega _k} z_k u_2 ^q   \textnormal{d} \vartheta _{\alpha  } \\
& \leq    C_5 C_\ast ^{p     -1 } \left[\frac{m_t}{r (t-s)} \right]^{p   } |\Omega _k ^\prime | \|u_2\| ^p _{1,p,\mathbb{R}^d ,\alpha }.
\end{aligned}
\end{equation}

$\it{Steep \ 2.}$ \  From \eqref{15}, we see
\begin{equation}  \label{16}
\begin{aligned}
& \varepsilon  C_\ast ^{p   - 1}\left(   \int_{\Omega _k ^\prime} \hat{z}_k \textnormal{d}x^\prime   \right)^{p }  r^{ -p \frac{(N-d)(\gamma _1  - \gamma _2  )}{\gamma _1   \gamma _2  }} \|u_2\|_{p,\mathbb{R}^d,\alpha} ^p 
  + (1-\varepsilon) C_\ast ^{q}k^q 
 |\Omega  _k ^\prime |^{p\frac{\gamma _2   -  1}{\gamma _2  }}  \int _{\Omega _k ^\prime} \hat{z}_k \textnormal{d}x^\prime \|u_2\|_{q+1,\mathbb{R}^d,\alpha }^{q+1}  \\
&\leq    C_8 C_\ast ^{p     -1 } \left[\frac{m_t}{r (t-s)} \right]^{p    } |\Omega _k ^\prime |^{1+  p  \frac{\gamma _2   -  1}{\gamma _2  }} \|u_2\|^p _{1,p,\mathbb{R}^d ,\alpha }.
\end{aligned}
\end{equation}
 Using   Young's inequality ($a^{\varepsilon} b^{1-\varepsilon} \leq \varepsilon a+(1-\varepsilon) b$) in the left-hand side of \eqref{16}, we obtain
\begin{equation*} 
\begin{aligned}
& k^{q    (1-\varepsilon)}     C_\ast ^{(p   - 1)\varepsilon +q   (1-\varepsilon)} 
 r^{ -p \frac{(N-d)(\gamma _1  - \gamma _2  )}{\gamma _1  \gamma _2  }\varepsilon} \left( \int_{\Omega _k ^\prime} \hat{z}_k \textnormal{d}x^\prime  \right)^{ 1 + (p-1) \varepsilon} \\
& \leq   C_{9}  C_\ast ^{p -1 } \left[\frac{m_t}{r (t-s)} \right]^{p}    |\Omega ^\prime _k|^{ 1+ p  \frac{\gamma _2   -  1}{\gamma _2  } \varepsilon}.
\end{aligned}
\end{equation*}
Therefore,
\begin{equation} \label{17}
\begin{aligned}
&  k^{\frac{q     (1-\varepsilon)}{\gamma _3}}  C_\ast ^{\frac{(  q      - p   +1)(1-\varepsilon )}{\gamma _3  }} \\
& \leq   C_{9}^{\frac{1}{\gamma _3  }} \left(   \int_{\Omega _k ^\prime } \hat{z}_k \textnormal{d}x^\prime  \right)^{-\frac{ 1 + (p- 1)\varepsilon }{\gamma _3  }} 
  r^{p \frac{(N-d)(\gamma _1  - \gamma _2  )}{\gamma _1  \gamma _2  }\frac{\varepsilon}{\gamma _3  }} \left[\frac{m_t}{r (t-s)} \right]^{\frac{p   }{\gamma _3  }}   |\Omega ^\prime _k| ,
\end{aligned}
\end{equation}
where $\gamma _3   =  1 +p \frac{\gamma _2   -  1}{\gamma _2  }\varepsilon $.

$\it{Steep \ 3.}$ \  Integrating \eqref{17} with respect to $k$,
\begin{equation*} 
\begin{aligned}
& C_\ast  ^{\frac{ (  q   - p  +1)(1-\varepsilon) }{\gamma _3  }}   \int _0 ^{\mathcal{K}} k^{\frac{q   (1-\varepsilon)}{\gamma _3  }}   \textnormal{d}k \\
& \leq  C_{9}  ^{\frac{1}{\gamma _3  }}  r^{ p \frac{(N-d)(\gamma _1  - \gamma _2  )}{\gamma _1  \gamma _2  }\frac{\varepsilon}{\gamma _3  }} \left[\frac{m_t}{r (t-s)} \right]^{\frac{p   }{\gamma _3  }}  
 \int _0 ^{\mathcal{K}} \left(   \int_{\Omega  _k ^\prime } \hat{z} _k \textnormal{d}x^\prime \right)^{-\frac{1 + (p-1) \varepsilon }{\gamma _3  }}  |\Omega ^\prime _k| \textnormal{d}k.
\end{aligned}
\end{equation*} 
Let us consider the equality  
$$
\frac{\textnormal{d}}{\textnormal{d}k} \left( \int_{\Omega  _k ^\prime } \hat{z}_{k} \textnormal{d}x^\prime \right) = -|\Omega ^\prime _k|.
$$
Since $1-\frac{ 1 + (p-1) \varepsilon }{\gamma _3  }>0$ and  $\mathcal{K}>1$, 
\begin{equation*}
\begin{aligned}
&\mathcal{K} ^{1+ \frac{ q   (1-\varepsilon)}{\gamma _3  } } C_\ast  ^{\frac{ (q     - p   +1)(1-\varepsilon )}{\gamma _3  }} 
\leq  C_{10}   r^{ p \frac{(N-d )(\gamma _1  - \gamma _2  )}{\gamma _1  \gamma _2  }\frac{\varepsilon}{\gamma _3  }}       \left[\frac{m_t}{r (t-s)} \right]^{\frac{p   }{\gamma _3  }}  
  \left(   \int_{\Omega  _0 ^\prime} \hat{z}_0 \textnormal{d}x^\prime   \right)^{1-\frac{  1 +(p-1) \varepsilon }{\gamma _3  }}.
\end{aligned}
\end{equation*}
Notice that $\mathcal{K}\geq m_s$. Apply the estimate
$$
\int_{\Omega  _0 ^\prime } \hat{z}_{0} \textnormal{d}x^\prime \leq m_{t} | V_{\ell , tr} ^{\prime} |,
$$ 
we have
\begin{equation*}
\begin{aligned}
&m_s ^{1+ \frac{ q (1-\varepsilon)}{\gamma _3  } } C_\ast  ^{\frac{( q   - p   +1)(1-\varepsilon) }{\gamma _3  }} \\
&\leq   C_{11}   r^{ p\frac{(N-d)(\gamma _1  - \gamma _2  )}{\gamma _1  \gamma _2  }\frac{\varepsilon}{\gamma _3  }}       \left[\frac{m_t}{r (t-s)} \right]^{\frac{p    }{\gamma _3  }}  
  m_t ^{1-\frac{ 1 + (p-1)\varepsilon }{\gamma _3  }} r^{(N-d)\left[ 1-\frac{ 1 + (p-1)\varepsilon }{\gamma _3  }\right]}.
\end{aligned}
\end{equation*}
Take $C_\ast>0$ such that
\begin{align*}
& C_\ast  ^{\frac{( q   - p   +1)(1-\varepsilon) }{\gamma _3  }} 
  = r^{ p  \frac{(N-d )(\gamma _1  - \gamma _2  )}{\gamma _1  \gamma _2  }\frac{\varepsilon}{\gamma _3  } -\frac{p   }{\gamma _3  } + (N-d)\left[ 1-\frac{ 1 + (p- 1)\varepsilon }{\gamma _3  }\right]}.
\end{align*}
Hence $C_\ast = r^{-\tau}$, where
\begin{align*}
& \tau  = 
-\left\{ p \frac{(N-d)(\gamma _1  - \gamma _2  )}{\gamma _1  \gamma _2  }\frac{\varepsilon}{\gamma _3  } - \frac{p   }{\gamma _3  } \right.
  \left.+ (N-d)\left[ 1-\frac{1 + (p-1) \varepsilon}{\gamma _3  }\right]\right\}\\
& \cdot\left[\frac{ ( q   - p   +1)(1-\varepsilon )}{\gamma _3  }\right]^{-1}
 \in \left( \frac{p  }{q     -p    +1 } , \frac{p  }{q     -p    +1  } +\varepsilon _0 \right).
\end{align*}
Thus we get
$$
m_{s} \leq C_{11}^{\frac{\gamma _3  }{p  }\sigma} \frac{m_{t}^{\theta}}{(t-s)^{\sigma}},
$$
where
\begin{align*}
\theta=\left[\frac{p   }{\gamma _3  }+ 1-\frac{1 + (p-1) \varepsilon }{\gamma _3  }\right] \left[1+ \frac{q  (1-\varepsilon)}{\gamma _3  }\right]^{-1} < 1\quad \text { and }\quad
\sigma=\frac{p    }{\gamma _3  } \left[1+ \frac{q   (1-\varepsilon)}{\gamma _3  }\right]^{-1}. 
\end{align*}
By virtue of Lemma \ref{64}, we derive 
$$
m_{1 /2} \leq C_{12} .
$$
From the substitution $u_1=C_\ast w_1$ we obtain
$$
\operatorname{ess} \operatorname{sup} \left\{u_1(x^\prime)\:|\: x \in V_{\ell , r/2} \cap \Omega _+ \right\}=C_\ast m_{1/2} \leq C_{12} C_\ast =  C_{12} r^{-\tau}.
$$
Therefore, we  conclude  the proof of the proposition. \end{proof}

Proceeding similarly to  Lemma \ref{19} and Proposition \ref{66}, we have the following results.
\begin{lemma} \label{69}  (Without the  boundary condition)  Let $\Omega = \mathbb{R}^{d} \times B _R$ and $\Gamma = \{x\in \mathbb{R} ^N \:|\: x_{d+1} = \cdots =x_N=0\}$, where $B_R=\{(x_{d+1}, $ $\ldots, x_N)\in \mathbb{R}^{N-d}\:|\: |(x_{d+1}, \ldots, x_N)| <R\}$, $R\in (0,\infty]$,  $N\geq 3$ and $d\in \{ 1, \ldots, N-2 \}$. If $d=0$, we write $\Omega = B_R$ and $\Gamma = \{x\in \mathbb{R} ^N \:|\: x=0\}$. Assume that \ref{43} is satisfied and $1<p<N-d$. Under the assumption $(A_2)$, suppose that $u \in W_{\operatorname{loc}  }^{1, p  }(\bar{\Omega}  \backslash \Gamma ; \vartheta _{\alpha}) \cap L_{\operatorname{loc}  }^{\infty}(\bar{\Omega}  \backslash \Gamma)$ satisfies  
\begin{equation} \label{72} 
\int_{\Omega}\left( \left\langle A(\nabla u) , \nabla \varphi \right\rangle +  a( u) \varphi  +  g (u) \varphi  \right)   \textnormal{d} \vartheta _{\alpha  } \leq 0,
\end{equation}
for all $\varphi \in W_{\operatorname{loc}  }^{1, p  }(\bar{\Omega}  \backslash \Gamma ; \vartheta _{\alpha}) \cap L_{\operatorname{loc}  }^{\infty}(\bar{\Omega}  \backslash \Gamma)$,  $\varphi \geq 0$, with $\operatorname{supp}\varphi \subset \bar{\Omega}  \backslash \Gamma$. Then, if $\varepsilon _0 >0$ and  $0<r<\ell<r_0$ we have the estimate
$$
\max \{u_1 , 0\} \leq C r^{-\tau}, \quad \text { \textit{a.e.} in } \{x^\prime \in B_R\:|\: ||x^\prime| - \ell| <r\},
$$
where $C=C\left(N, \mu, \alpha  , d , p , q , \|u_2 \|_{q+1,\mathbb{R}^d,\alpha} , \|u_2\| _{1,p, \mathbb{R}^d,\alpha}     , \mathcal{U} \right)>0$ and $\tau=\tau\left(N, \mu , \alpha ,  d , \varepsilon _0 ,  p , q   \right)\in (\frac{p}{q    - p+1} , \frac{p}{q    - p+1} +\varepsilon _0 )$.
\end{lemma}

\begin{proposition}   (Without the  boundary condition)   Assume that \ref{43} and \ref{45} are satisfied and $1<p<N-d$. Under the assumption $(A_2)$, suppose that $u \in W_{\operatorname{loc}  }^{1  , p  }( \bar{\Omega}  \backslash \Gamma ; \vartheta _{\alpha}) \cap L_{\operatorname{loc}  }^{\infty}(\bar{\Omega } \backslash \Gamma)$ is a (weak) solution  in $\bar{\Omega } \backslash \Gamma$ of equation \eqref{71}. Then, if  $\varepsilon _0 >0$, we have the estimate
$$
|u_1(x^\prime)| \leq C |x^\prime|  ^{-\tau}, \quad \text { \textit{a.e.} in } \left\{x^\prime \in B_R \:|\: 0<|x^\prime|  <r_0\right\},
$$
where $C=C\left(N,  \mu,\alpha , d , p, q  , \|u_2 \|_{q+1,\mathbb{R}^d,\alpha} , \|u_2\| _{1,p, \mathbb{R}^d,\alpha}    , \mathcal{U} \right)>0$ and  $\tau=\tau\left(N, \mu, \alpha , d , \varepsilon _0 , p, q     \right)\in (\frac{p}{q    - p+1} , \frac{p}{q    - p+1} +\varepsilon _0 )$.
\end{proposition}

\subsection{Proof of Proposition \ref{266}} The proof follows from the next lemma.

\begin{lemma} \label{219}
Assume that the condition  \ref{43} is  satisfied, $p   -1 < \beta $ and   $\mu _1>0$. Under the assumption $(A_1)$, suppose that   $u\in  W^{1, p  } _{\operatorname{loc}} (\bar{\Omega} \backslash \Gamma  )$ $\cap L^{\infty} _{\operatorname{loc}} (\bar{\Omega}  \backslash \Gamma )$ satisfies 
\begin{equation} \label{222}
\int _{\Omega }\left( \left\langle A(\nabla u) , \nabla \varphi \right\rangle +  g (u) \varphi  \right) \textnormal{d}x  + \int _{\partial \Omega} b( u)   \varphi \textnormal{d}\vartheta _{\beta} \leq 0,
\end{equation}
for all $\varphi \in W_{\operatorname{loc}}^{1, p }(\bar{\Omega}  \backslash \Gamma ) \cap L_{\operatorname{loc}}^{\infty}(\bar{\Omega}  \backslash \Gamma )$, $\varphi \geq 0$, with   $\operatorname{supp} \varphi  \subset \bar{\Omega } \backslash \Gamma $. Then, if $\varepsilon _0 >0 $ and  $0<r<\ell < r_0$ we have the estimate
\begin{equation} \label{242}
\max \{u ,0\}   \leq C r^{-\tau}, \quad \text { \textit{a.e.} in } V_{\ell, r/2},
\end{equation}
where $C=C(N,  \mu   , \beta , d , \varepsilon _0 , p , q  ,  C_0 ,\mathcal{U} )>0$ and  $\tau=\tau (N,  \mu  ,  \beta , d ,  \varepsilon _0 , p , q   )\in ( \frac{p}{q- p   +1} ,  \frac{p}{q- p   +1} +\varepsilon _0)$. 
\end{lemma} 

\renewcommand*{\proofname}{Sketch of the proof}
\begin{proof} 

We follow the proof of Lemma \ref{19}. Let $u = C_\ast w $, where  $C_\ast>0$ is a  number that will be determined below. We assume that $|\{x \in V_{\ell,r/2}  \:|\: w(x)>0\} |\neq 0$, otherwise \eqref{242} is immediate.  Set $\Omega _+   =\{x \in V_{\ell,r}  \:|\: w(x)>0\}$. Take
$$
m_{t}=\operatorname{ess} \operatorname{sup} \left\{  w(x)\:|\: x \in V_{\ell,t r} \cap \Omega _+   \right\}, \quad 1/2 \leq t \leq 1 .
$$
 Let $1/2 \leq s<t \leq 1$.  Define the functions $z : \Omega  \rightarrow \mathbb{R}$, $z_k  :\Omega  \rightarrow \mathbb{R}$, by
\begin{align*}
z (x)&=   w(x) - m_{t}   \xi \left(\left|  |x^\prime|   - \ell \right|\right)  ,\\
z_k  (x)&=\left\{ \begin{aligned}
& \max \left\{ w(x)-   \left( m _{t}   \xi \left(\left|  |x^\prime|  - \ell \right|\right) + k\right) , 0\right\}  & & \text { if } x\in  V_{\ell , tr},\\
& 0 & & \text { if } x\in \Omega  \backslash V_{\ell , tr},
\end{aligned}
\right. 
\end{align*}
where  $0 \leq k\leq \operatorname{ess} \operatorname{sup} _{\Omega _+ } z$, and $\xi : \mathbb{R} \rightarrow \mathbb{R}$ is a smooth function satisfying: $\xi = 0$ on $(-\infty , s r]$, $\xi  = 1$ on $ \left[ \frac{s+t}{2} r , \infty \right)$,
$$
0 \leq \xi \leq 1 \quad \text { and } \quad \left| \xi^{\prime}\right| \leq \frac{C_1}{r(t-s)} \quad \text { on  }    \mathbb{R}.
$$

\noindent By $(A_1)$, we have
\begin{align}
&|\operatorname{supp} z_k|\leq C_2 r^{N-d}\left|\left\{ x^{\prime \prime}\in \mathbb{R} ^{d}\:|\:  \frac{C_0 m _t}{k}\geq (1+ |x^{\prime \prime}|)^{\delta} \right\}\right| <\infty, \quad \text { if } k>0,\label{27}\\
& \int_{\{x \in V_{\ell, t r}\:|\: z_0 (x)>0\}} z_0 \textnormal{d}x \leq \int  _{\{x \in V_{\ell, t r}\:|\: z_0 (x)>0\}} w \textnormal{d}x \leq C_{2} m_t r^{N-d}. \nonumber
\end{align}

\noindent Take $k \in (0, \mathcal{K} )$, where $\mathcal{K} = \sup \{k\in [0 , \operatorname{ess} \operatorname{sup} _{\Omega   ^\prime} z] \:|\:  |\{ x\in V_{\ell , tr} \:|\: z_k (x) >0\} |\neq 0\}$. Observe that $\mathcal{K} \geq m_s \geq m_{1/2} $. Substituting $\varphi =  z_k$ into \eqref{222}, we obtain
\begin{equation}\label{237}
\begin{aligned} 
&\int_{\Omega  _k }    C_{\ast} ^{p - 1}|\nabla w|^{p}    \textnormal{d}x   +   \int_{\partial \Omega  _k \cap \partial \Omega }   C_\ast ^{p-1}  z_k ^p   \textnormal{d}\vartheta _{\beta}  
+ \int_{\Omega  _k}    C_\ast^{q   } k^{q   } z_k   \textnormal{d}x \\
&\leq  C_3   \int_{\Omega  _k } \frac{m_t C_\ast ^{p-1}}{r(t-s)}  | \nabla w|^{p-1}  \textnormal{d}x,
\end{aligned}
\end{equation}
where $\Omega _k  =\{x \in V_{\ell, t r}\:|\: z_k (x)>0\}$. Using Young's inequality, we obtain
\begin{equation} \label{238}
\begin{aligned}
  \int_{\Omega  _k }  \frac{  m _t  C_\ast ^{p-1}}{r(t-s)}   |\nabla w|^{p-1} \textnormal{d}x  
\leq   C_\ast ^{p-1} C_4 \left[\frac{m _t}{r(t-s)}\right]^{p} |\Omega  _k | + \varepsilon _1 \int _{\Omega _k}C_\ast ^{p-1} |\nabla w|^{p}   \textnormal{d}x .
\end{aligned}
\end{equation}
\noindent Take  $\varepsilon _1 = \frac{1}{2C_2 }$. Then, by \eqref{237} and \eqref{238},
\begin{align*}
\frac{1}{2}\int_{\Omega _k   }     C_{\ast} ^{p- 1}  |\nabla w|^{p}    \textnormal{d}x  +  \int_{\partial \Omega  _k \cap \partial \Omega} C_\ast ^{p-1} z_k ^p \textnormal{d}\vartheta _{\beta}  
+ \int_{\Omega  _k } C_\ast ^{q   } k^{q   }z_k  \textnormal{d}x  \leq    C_5  C_\ast ^{p   -1 } \left[\frac{m_t}{r (t-s)} \right]^{p  }  |\Omega _k|.
\end{align*}
Note that $\nabla z_{k}=\nabla w-m_{t}  \xi ^{\prime}  \nabla | |x^\prime|   - \ell |$ in $\Omega  _k $. Then, 
\begin{equation} \label{214}
\begin{aligned}
C_{\ast} ^{p - 1} \int _{\Omega  _k }    |\nabla z_{k}|^{p  } \textnormal{d}x  +  C_{\ast} ^{p - 1}\int _{\partial \Omega _k \cap \partial \Omega}  z_k ^{p  }  \textnormal{d}\vartheta _{\beta}  
  + C_\ast ^{q   } k^{q }\int_{\Omega  _k } z_k  \textnormal{d}x \leq    C_6  C_\ast ^{p -1 } \left[\frac{m_t}{r (t-s)} \right]^{p  } |\Omega _k | .
\end{aligned}
\end{equation}
Choose  $\gamma _1  , \gamma _2  >0$ and $\varepsilon \in (0,1)$  such that $p <\gamma _2 <\gamma _1 $, $\frac{N}{p} - \frac{N}{\gamma _1} < 1$, $
\frac{p  }{q     -p    +1 } + \varepsilon_0 > \frac{ p}{(q   - p   +1)(1-\varepsilon)}$ and $1 > (N -d)\left( 1- \frac{p}{\gamma _1}  \right) $.  Then, by Propositions \ref{68} and \ref{59}, and \eqref{27}, we have
\begin{equation}\label{239} 
\begin{aligned}
&\left(\int_{\Omega  _k}z_{k}^{\gamma _2  } \textnormal{d}x \right)^{\frac{1}{\gamma _2  }}\leq \left(\int _{\Omega  _k} z_k ^{\gamma _1} \textnormal{d}x \right)^{\frac{1}{\gamma _1}} |\Omega _k|^{\frac{\gamma _1 - \gamma _2}{\gamma _1 \gamma _2}} \\  
&\leq  C_7 \left(\frac{m_t}{k}\right)^{\frac{d(\gamma _1 - \gamma _2)}{\delta\gamma _1 \gamma _2}} r^{\frac{(N-d)(\gamma _1   - \gamma _2  )}{\gamma _1  \gamma _2  }} \left(\int _{\Omega  _k}\left|\nabla z_{k}\right|^{p} \textnormal{d}x +  \int _{ \partial \Omega _k \cap \partial \Omega} z_{k}^{p} \textnormal{d}\vartheta _{\beta}   \right)^{\frac{1}{p }}.
\end{aligned}
\end{equation}
Using Hölder's inequality, we get
\begin{align}
&\int_{\Omega _k} z_{k} \textnormal{d}x  \leq \left(\int_{\Omega  _k} z_{k}^{ \gamma _2  } \textnormal{d}x \right)^{\frac{1}{\gamma _2  }}|\Omega  _k |^{\frac{\gamma _2   -1}{ \gamma _2   }} \label{240}.
\end{align} 
From \eqref{214} - \eqref{240}, we have
\begin{equation} \label{215}
\begin{aligned}
&  C_\ast  ^{p   - 1} \left[ |\Omega  _k | ^{-\frac{\gamma _2   -  1}{\gamma _2  }}  \left(\frac{m_t}{k}\right)^{-\frac{d(\gamma _1 - \gamma _2)}{\delta \gamma _1 \gamma _2}}r^{-\frac{(N-d)(\gamma _1  - \gamma _2  )}{\gamma _1  \gamma _2  }}   \int_{\Omega _k} z_k \textnormal{d}x  \right] ^{p } 
  +  C_\ast ^{q} k^q    \int _{\Omega _k} z_k \textnormal{d}x \\
& \leq    C_8 C_\ast ^{p     -1 } \left[\frac{m_t}{r (t-s)} \right]^{p   } |\Omega _k | .
\end{aligned}
\end{equation}

\noindent Therefore,
\begin{equation} \label{217}
\begin{aligned}
&  k^{\frac{q     (1-\varepsilon)}{\gamma _3}}  C_\ast ^{\frac{(  q      - p   +1)(1-\varepsilon )}{\gamma _3  }} \\
& \leq   C_{9} \left(   \int_{\Omega _k } z_k \textnormal{d}x  \right)^{-\frac{ 1 + (p- 1)\varepsilon }{\gamma _3  }} 
 \left(\frac{m_t}{k}\right)^{p \frac{d(\gamma _1  - \gamma _2  )}{\delta\gamma _1  \gamma _2  }\frac{\varepsilon}{\gamma _3  }}  r^{p \frac{(N-d)(\gamma _1  - \gamma _2  )}{\gamma _1  \gamma _2  }\frac{\varepsilon}{\gamma _3  }} \left[\frac{m_t}{r (t-s)} \right]^{\frac{p   }{\gamma _3  }}   |\Omega _k| ,
\end{aligned}
\end{equation}
where $\gamma _3   =  1 +p \frac{\gamma _2   -  1}{\gamma _2  }\varepsilon $ and $k>0$. Finally, through the argument  in the proof of Lemma \ref{19}, we  conclude the proof. \end{proof}

Proceeding in the same way as in  Lemma \ref{219} and Proposition \ref{266}, we have the following results.

\begin{lemma} \label{269} (Without the  boundary condition)  Let $\Omega = \mathbb{R}^{d} \times B _R$ and $\Gamma = \{x\in \mathbb{R} ^N \:|\: x_{d+1} = \cdots =x_N=0\}$, where $B_R=\{(x_{d+1}, $ $\ldots, x_N)\in \mathbb{R}^{N-d}\:|\: |(x_{d+1}, \ldots, x_N)| <R\}$, $R\in (0,\infty]$,  $N\geq 3$ and $d\in \{ 1, \ldots, N-2 \}$. If $d=0$, we write $\Omega = B_R$ and $\Gamma = \{x\in \mathbb{R} ^N \:|\: x=0\}$. Assume that \ref{43} is satisfied and $\mu_1 >0$. Under the assumption $(A_1)$, suppose that $u \in W_{\operatorname{loc}  }^{1, p  }(\bar{\Omega}  \backslash \Gamma) \cap L_{\operatorname{loc}  }^{\infty}(\bar{ \Omega } \backslash \Gamma)$ satisfies  
\begin{equation} \label{272}
\int_{\Omega}\left( \left\langle A(\nabla u) , \nabla \varphi \right\rangle +  a( u) \varphi  +  g (u) \varphi  \right)   \textnormal{d}x \leq 0,
\end{equation}
for all $\varphi \in W_{\operatorname{loc}  }^{1, p  }(\bar{ \Omega } \backslash \Gamma) \cap L_{\operatorname{loc}  }^{\infty}(\bar{ \Omega  }\backslash \Gamma)$,  $\varphi \geq 0$, with $\operatorname{supp}\varphi \subset \bar{ \Omega } \backslash \Gamma$. Then, if $\varepsilon _0 >0$ and  $0<r<\ell<r_0$ we have the estimate
$$
\max \{u , 0\} \leq C r^{-\tau} \quad\text{\textit{a.e.} in}\quad V_{\ell, r / 2} ,
$$
where $C=C\left(N, \mu , d , p , q , C_0 ,\mathcal{U} \right)>0$ and $\tau=\tau\left(N, \mu,  d , \varepsilon _0 ,  p , q   \right)\in (\frac{p}{q    - p+1} , $ $ \frac{p}{q    - p+1} +\varepsilon _0 )$.
\end{lemma}

\begin{proposition}  (Without the  boundary condition)    Assume that \ref{43} and \ref{45} are satisfied and $\mu_1>0$. Under the assumption $(A_1)$, suppose that $u \in W_{\operatorname{loc}  }^{1  , p  }( \bar{ \Omega } \backslash \Gamma) \cap L_{\operatorname{loc}  }^{\infty}(\bar{ \Omega  }\backslash \Gamma)$ is a (weak) solution  in $\bar{ \Omega } \backslash \Gamma$ of equation \eqref{271}. Then, if  $\varepsilon _0 >0$, we have the estimate
$$
|u(x)| \leq C |x|  ^{-\tau} \quad\text{\textit{a.e.} in}\quad \left\{x \in \mathbb{R}^{N-d} \:|\: 0<|x^\prime|  <r_0\right\},
$$
where $C=C\left(N,  \mu, d, p, q  , C_0    , \mathcal{U}\right)>0$ and  $\tau=\tau\left(N, \mu, d, \varepsilon _0 , p, q     \right)\in (\frac{p}{q    - p+1} , \frac{p}{q    - p+1} +\varepsilon _0 )$.
\end{proposition}

\section{The removability of singular set} \label{33}

In this section we prove the main results of this paper. 

\subsection{Proof of  Theorem \ref{77}}

Before we turn to the proof of this theorem, we show the following auxiliary result. Similar to \cite{skrypnik2005removability}, we employ a logarithmic type test function.

\begin{lemma} \label{67}
Suppose  the  hypotheses as in  Lemma \ref{19}  and $N-d >\frac{pq}{q-p+1}$. Then
$$
\max \{u, 0\} \in L^{\infty}(\Omega) .
$$
\end{lemma}

\renewcommand*{\proofname}{Proof}

\begin{proof}
We proceed by contradiction. For $r \in\left(0, r_0^2\right)$, we denote
$$
\Lambda(r)=\operatorname{ess} \sup \left\{\max \{  u(x), 0\} \:|\: r \leq |x^\prime|   \leq r_0^2, \ x \in \Omega  \right\} .
$$
We have $\lim _{r \rightarrow 0^{+}} \Lambda(r)=\infty$. For sufficiently small values $r$ we define the function $\psi_r: \mathbb{R} \rightarrow \mathbb{R}$ as follows:
$$
\psi_r(t)= \begin{cases}0 & \text { if } t<r, \\ 1 & \text { if } t>\sqrt{r}, \\ \frac{2}{\ln \frac{1}{r}} \ln \frac{t}{r} & \text { if } r \leq t \leq \sqrt{r} .\end{cases}
$$
Choosing $\rho>0$ such that $\Lambda(\rho)>1$, set
$$
\varphi(x)=\left(\ln \ \max \left\{\frac{   u(x)  }{\Lambda(\rho)}, 1\right\}\right)  \psi_r^\gamma (|x^\prime|) ,
$$
where $\gamma= \frac{p q }{q  -p+1}$. We have $\varphi \in W_{\operatorname{loc}  }^{1, p  }(\bar{\Omega}   \backslash \Gamma ; \vartheta _{\alpha}) \cap L_{\operatorname{loc}  }^{\infty}(\bar{\Omega}   \backslash \Gamma)$, with $\operatorname{supp}\varphi \subset \bar{\Omega}   \cap\{|x^\prime|   \geq r\} \subset \bar{\Omega}   \backslash \Gamma$. For simplicity we write $\psi_r(|x^\prime|)=\psi_r (x)$ and $\psi_r^{\prime} ( |x^\prime| ) =\psi_r^{\prime} (x)$. Set $\Omega  _\rho=\{x \in \Omega   \:|\:    u(x^\prime )>\Lambda(\rho)\}$. Substituting $\varphi$ into \eqref{22}, we obtain
$$
\begin{aligned}
& \int_{\Omega_\rho} \left( \psi_r^\gamma \left\langle A(  \nabla u), \frac{1}{u}\nabla u  \right\rangle 
+   \gamma \psi_r^{\prime} \psi_r^{\gamma-1}\left(\ln \frac{u}{\Lambda(\rho)}\right)\langle A(  \nabla u), \nabla |x^\prime|  \rangle \right)  \textnormal{d}\vartheta _{\alpha} \\
&+\int _{\Omega  _\rho} g(u) \psi_r^\gamma\left(\ln \frac{ u}{\Lambda(\rho)}\right)  \textnormal{d} \vartheta _{\alpha } 
 +\int_{\partial \Omega  _\rho \cap \partial \Omega  }b(u)   \psi_r^\gamma\left(\ln \frac{u}{\Lambda(\rho)}\right) \textnormal{d}\vartheta _{\beta} \leq 0 .
\end{aligned}
$$
By virtue of  \ref{43}  and $u>\Lambda(\rho)$ in $\Omega  _\rho$, we have
$$
\begin{aligned}
& \int_{\Omega  _\rho} \mu \frac{\psi_r^\gamma}{u}|\nabla u|^{p}  \textnormal{d}\vartheta _{\alpha}+\int_{\Omega  _\rho} \psi_r^\gamma\left(\ln \frac{u}{\Lambda(\rho)}\right) u^{q}  \textnormal{d}\vartheta _{\alpha}   \\
& \leq \int_{\Omega  _\rho} \mu^{-1} \gamma \psi_r^{\prime} \psi_r^{\gamma-1}\left(\ln \frac{u}{\Lambda(\rho)}\right)|\nabla u|^{p   -1}  \textnormal{d}\vartheta _{\alpha}  \\
& \leq \int_{\Omega  _\rho} \mu^{-1} \gamma \psi_r^\gamma u^{-1}\left\{C_1\left(\varepsilon_1, p   \right)\left[\psi_r^{\prime} \psi_r^{-1} u\left(\ln \frac{u}{\Lambda(\rho)}\right)\right]^{p}+\varepsilon_1|\nabla u|^{p}\right\}  \textnormal{d}\vartheta _{\alpha}  .
\end{aligned}
$$
Take $\varepsilon_1=\frac{\mu^2}{2 \gamma }$,
\begin{equation} \label{74}
\begin{aligned}
& \int_{\Omega  _\rho} \frac{\psi_r^\gamma}{u}|\nabla u|^{p}  \textnormal{d}\vartheta _{\alpha}  +\int_{\Omega  _\rho} \psi_r^\gamma\left(\ln \frac{u}{\Lambda(\rho)}\right) u^{q} \textnormal{d}\vartheta _{\alpha}   \\
& \leq C_2 \left(\mu, \gamma, p   , q     \right) \int_{\Omega  _\rho}\left(\psi_r^{\prime}\right)^{p} \psi_r^{\gamma-p   } u^{p  -1}\left(\ln \frac{ u}{\Lambda(\rho)}\right)^{p}  \textnormal{d}\vartheta _{\alpha} .
\end{aligned}
\end{equation}
Additionally, let us consider
$$
\begin{aligned}
& \left(\psi_r^{\prime}\right)^{p} \psi_r^{\gamma-p   } u^{p -1}\left(\ln \frac{u}{\Lambda(\rho)}\right)^{p} \leq  \left(\ln \frac{u}{\Lambda(\rho)}\right)\\
&\cdot \left\{C_3\left(\varepsilon_2, p   , q     \right)\left[\left(\psi_r^{\prime}\right)^{p}\left(\ln \frac{u}{\Lambda(\rho)}\right)^{p  -1}\right]^{\frac{q     }{q  - p   +1}}+\varepsilon_2 \psi_r^{\frac{(\gamma - p)    q  }{p   -1}} u^{q }\right\},
\end{aligned}
$$
$$
 \frac{\left(\gamma-p   \right) q     }{p   -1} \geq \gamma \quad \text { and } \quad \psi_r^{\prime}(t)=\frac{2}{t \ln \frac{1}{r}}>1 \text { for } t \in(r, \sqrt{r}) .
$$
\noindent Choose $\varepsilon_2=\frac{1}{2 C_2}$. By \eqref{74},
$$
\begin{aligned}
& \int_{\Omega  _\rho} \frac{\psi_r^\gamma}{u}|\nabla u|^{p} \textnormal{d}\vartheta _{\alpha}  +\int_{\Omega  _\rho} \psi_r^\gamma\left(\ln \frac{u}{\Lambda(\rho)}\right) u^{q} \textnormal{d}\vartheta _{\alpha}     \\
& \leq C_4 \int_{\Omega  _\rho \cap \{ r \leq |x^\prime| \leq \sqrt{r}\}}\left(\ln \frac{u}{\Lambda(\rho)}\right)\left(\ln \frac{u}{\Lambda(\rho)}\right)^{\frac{\left(p-1\right) q     }{q-p   +1}}\left(\frac{2}{|x^\prime|   \ln \frac{1}{r}}\right)^{\frac{p q}{q     -p   +1}}\textnormal{d}\vartheta _{\alpha} ,
\end{aligned}
$$
where $C_4 = C_4(\mu, \alpha ,  \gamma, p   , q  )>0$. Take $\delta _1>2$ and  $\tau > \frac{p}{q-p+1}$ such that $\frac{(N-d-\gamma)(q-p+1)}{2p} > \frac{N-d -\gamma}{\delta _1 \tau} >q_1$. Using \eqref{75}, we see
\begin{align*}
& \int_{\Omega  _\rho} \frac{\psi_r^\gamma}{u}|\nabla u|^{p}  \textnormal{d}\vartheta _{\alpha}  +\int_{\Omega  _\rho} \psi_r^\gamma\left(\ln \frac{ u}{\Lambda(\rho)}\right) u^{q} \textnormal{d}\vartheta _{\alpha}   \\
& \leq C_5\left(\ln \frac{1}{r}\right)^{-\frac{p     q    }{q -p +1}} \int_{\Omega  _\rho \cap \{r \leq |x^\prime| \leq \sqrt{r}\}}    \left[ \ln  \left( \frac{u}{\Lambda (\rho)  }\right) ^{  \frac{(N-d-\gamma)/(\tau \delta _1)}{ 1+(p     -1) q / (q-p   +1 )}}  \right]^{1+\frac{(p     -1) q }{q-p   +1}}  |x^\prime|   ^{-\frac{pq }{q-p   +1}}  \textnormal{d}\vartheta _{\alpha}\\
& \leq C_6 \left(\ln \frac{1}{r}\right)^{-\frac{p     q    }{q -p +1}} \int_{\Omega  _\rho \cap \{r \leq |x^\prime| \leq \sqrt{r}\}}    u ^{  \frac{N-d-\gamma}{\tau \delta _1}}    |x^\prime| ^{-\frac{pq }{q-p   +1}}  \textnormal{d}\vartheta _{\alpha}\\
& \leq C_7 \left(\ln \frac{1}{r}\right)^{-\frac{p     q    }{q -p +1}} \int_{\Omega  _\rho \cap \{r \leq |x^\prime| \leq \sqrt{r}\}}    |x^{\prime}| ^{  -\frac{N-d-\gamma}{\delta _1}}  \left( \frac{u_2}{\|u_2\|_{L^{\infty} (\mathbb{R}^d) }} \right)^{q_1}|x^\prime| ^{-\frac{pq }{q-p   +1}}  \textnormal{d}\vartheta _{\alpha}.
\end{align*}
Also, by $(A_2)$,
\begin{align*}
& \leq C_8\left(\ln \frac{1}{r}\right)^{-\frac{p     q    }{q -p +1}} \int_{\Omega ^\prime _\rho \cap \{r \leq |x^\prime| \leq \sqrt{r}\}}     r ^{  -\frac{N-d-\gamma}{\delta _1}}  |x^\prime| ^{-\frac{pq }{q-p   +1}}  \textnormal{d}x^\prime   \\
& \leq C_9\left(\ln \frac{1}{r}\right)^{-\frac{p    q    }{q  -p    +1}}  r ^{  -\frac{N-d -\gamma}{\delta _1}} \int_r^{\sqrt{r}}    t^{N-d -1-\gamma} \textnormal{d}  t \\
& =C_9\left(\ln \frac{1}{r}\right)^{-\frac{p     q    }{q -p +1}}  r ^{  -\frac{N-d -\gamma}{\delta _1}} \frac{r^{\frac{N -d-\gamma}{2}}}{N -d-\gamma }\left(1-r^{\frac{N-d-\gamma  }{2}}\right),
\end{align*}
where $\Omega  _\rho ^\prime =  \{x^\prime \:|\: x\in \Omega _\rho\}$ and $C_i>0$, $i=5,\ldots, 9$, are constants independent of $r$.  Therefore, if $r \rightarrow 0^{+}$,
$$
\int_{\Omega  _\rho} \frac{|\nabla u|^{p }}{u}  \textnormal{d}\vartheta _{\alpha}  +\int_{\Omega  _\rho}\left(\ln \frac{u}{\Lambda(\rho)}\right) u^{q}  \textnormal{d}\vartheta _{\alpha}  =0 .
$$
Hence, $u (x  )=\Lambda(\rho)$ \textit{a.e.} in $\Omega  _\rho  $. Thus, we have a contradiction, and proves  the lemma. \end{proof}

Similarly as in  Lemma \ref{67}  we have the following result.

\begin{lemma} (Without the  boundary condition)  Assume the hypotheses as in  Lemma \ref{69} and $N-d >\frac{pq}{q-p+1}$. Then
$$
\max \{u , 0\} \in L^{\infty}(\Omega ).
$$
\end{lemma}

We are now in a position to prove the theorem.

\renewcommand*{\proofname}{Proof of Theorem \ref{77}:}

\begin{proof}
$\it{Steep \ 1.}$ \  First we prove $u \in W^{1, p }(\Omega  ; \vartheta _{\alpha}) \cap L^{\infty}(\Omega )$. As consequence of Lemma \ref{67}, $u \in L^{\infty}(\Omega)$. Next, for $r<2 r_0 / 5$, let $\psi_r: \mathbb{R} \rightarrow \mathbb{R}$ be a smooth function such that
$$
\psi_r(t)=\left\{
\begin{aligned}
&0 & &\text { if } t<r / 2 \text { or } t>5 r / 2, \\
&1 & &\text { if } r<t<2 r,
\end{aligned}
\right.
$$
$0 \leq \psi_r \leq 1$ and $\left|\psi_r^{\prime}\right| \leq C / r$, where $C$ is a suitable positive constant. Set
$$
\varphi (x)=\psi_r^{p} (|x^\prime|  ) u(x).
$$
We have $\varphi \in W_{\operatorname{loc}  }^{1, p }(\bar{\Omega}   \backslash \Gamma ; \vartheta _{\alpha}) \cap L_{\operatorname{loc}  }^{\infty}(\bar{\Omega}   \backslash \Gamma)$, with $\operatorname{supp}\varphi \subset \bar{\Omega}   \cap\{r / 2 \leq |x^\prime|   \leq 5 r / 2\}$. For simplicity we write $\psi_r (x)=\psi_r ( |x^\prime| ) $ and $\psi_r^{\prime} (x)=\psi_r^{\prime} ( |x^\prime| ) $. Substituting $\varphi$ into \eqref{62},
$$
\int_{\Omega} \left(\left\langle A( \nabla u), p \psi_r^{p-1} \psi_r^{\prime} u \nabla |x^\prime|+\psi_r^{p} \nabla u\right\rangle 
 +g( u) \psi_r^{p} u \right) \textnormal{d}\vartheta _{\alpha}+\int_{\partial \Omega  } b(u) \psi_r ^{p    } u  \textnormal{d}\vartheta _{\beta}   \leq 0.
$$
By  \ref{43}, we have
$$
\begin{aligned}
& \int_{\Omega  } \mu|\nabla u|^{p} \psi_r^{p} \textnormal{d}\vartheta _{\alpha}  +\int_{\Omega }  \psi_r^{p}|u|^{q+1} \textnormal{d}\vartheta _{\alpha} +\int_{\partial \Omega  } \mu _1 \psi_r^{p}|u|^{p} \textnormal{d}\vartheta _{\beta}   \\
 &\leq \int_{\Omega  } \mu^{-1} p \psi_r^{p  -1}\left|\psi_r^{\prime}\right||u||\nabla u|^{p-1} \textnormal{d}\vartheta _{\alpha} .
\end{aligned}
$$
Then,
$$
\begin{aligned}
& \int_{\Omega  }|\nabla u|^{p } \psi_r^{p} \textnormal{d}\vartheta _{\alpha}   
 \leq \mu^{-2} p C_1 (\varepsilon_1, p )\int_{\Omega \cap \{r/2 \leq |x^\prime| \leq 5r/2\}  } \left(\left|\psi_r^{\prime}\right||u|\right)^{p}\textnormal{d}\vartheta _{\alpha } \\
&+ \mu ^{-2}p \varepsilon_1\int _{\Omega} \left(|\nabla u|^{p -1} \psi_r^{p-1}\right)^{\frac{p}{p-1}} \textnormal{d}\vartheta _{\alpha}  .
\end{aligned}
$$
Take $\varepsilon_1=\frac{\mu^2}{2 p }$. By Lemma \ref{67},
\begin{equation}\label{28}
\int_{\Omega   \cap \{r \leq |x^\prime| \leq 2 r \}}|\nabla u|^{p } \textnormal{d}\vartheta _{\alpha}   \leq C_2 r^{N-d-p}
\end{equation}
where $C_2$ is a positive constant independent of $r$. Therefore,
$$
\begin{aligned}
& \int_{\Omega   \cap\{|x^\prime| \leq 2 r\}}|\nabla u|^{p} \textnormal{d}\vartheta _{\alpha}   \\
& =\sum_{i=0}^{\infty} \int_{\Omega   \cap \{2^{-i} r \leq |x^\prime| \leq 2^{-i+1} r\}}|\nabla u|^{p} \textnormal{d}\vartheta _{\alpha}   \leq C_2 \sum_{i=0}^{\infty}\left(\frac{r}{2^i}\right)^{N-d-p    }<\infty .
\end{aligned}
$$
So $|\nabla u| \in L^{p }(\Omega ; \vartheta _{\alpha}  )$, and thus we have proved that $u \in W^{1, p }(\Omega , \vartheta _{\alpha} ) \cap L^{\infty}(\Omega  )$.\\

$\it{Steep \ 2.}$ \ Now, we will show that $u$ is a solution of \eqref{1} in the domain $\Omega $. For $r \in\left(0, r_0\right)$, let $\xi_r: \mathbb{R} \rightarrow \mathbb{R}$ be a smooth function such that
$$
\xi_r(t)= \begin{cases}1 & \text { if }|t| \leq r \\ 0 & \text { if } 2 r \leq|t|\end{cases}
$$
$0 \leq \xi \leq 1$ and $\left|\xi_r^{\prime}\right| \leq C / r$, where $C$ is a suitable positive constant. For simplicity we write $\xi_r (x)=\xi_r (|x^\prime| ) $ and $\xi_r^{\prime} (x)=\xi_r^{\prime} (|x^\prime|)  $. Let $\varphi \in W^{1, p }(\Omega ; \vartheta _{\alpha} ) \cap L^{\infty}(\Omega  )$. We have $\left(1-\xi_r (|x^\prime| ) \right) \varphi \in W_{\operatorname{loc}  }^{1, p }(\bar{\Omega}   \backslash \Gamma ; \vartheta _{\alpha}) \cap L_{\operatorname{loc}  }^{\infty}(\bar{\Omega}   \backslash \Gamma)$, with $\operatorname{supp}\varphi \subset \bar{\Omega}   \backslash \Gamma$. Then, \eqref{62} yields
\begin{equation} \label{76}
\begin{aligned}
& \int_{\Omega}\left(\left\langle A( \nabla u),\left(1-\xi_r\right) \nabla \varphi-\varphi \xi_r^{\prime} \nabla |x^\prime|  \right\rangle 
 +g( u)\left(1-\xi_r\right) \varphi \right) \textnormal{d}\vartheta _{\alpha}  \\ 
& +\int_{\partial \Omega  } b(u)  \left(1-\xi_r\right) \varphi \textnormal{d}\vartheta _{\beta}=0 .
\end{aligned}
\end{equation}
When $r \rightarrow 0^{+}$, for all $\varphi \in W^{1 , p }(\Omega ; \vartheta _{\alpha} ) \cap L^{\infty}(\Omega  )$, the equality \eqref{76} implies \eqref{62}. Indeed, we have
$$
\begin{aligned}
& \lim _{r \rightarrow 0} \int_{\Omega  }  \left( \left\langle A(  \nabla u),\left(1-\xi_r\right) \nabla \varphi\right\rangle  +g(u)\left(1-\xi_r\right) \varphi   \right) \textnormal{d}\vartheta _{\alpha}  + \int_{\partial \Omega} b( u) \left(1-\xi_r\right) \varphi \textnormal{d}\vartheta _{\beta} \\
& =\int_{\Omega}  \left( \langle A(\nabla u), \nabla \varphi\rangle + g(u) \varphi  \right) \textnormal{d}\vartheta _{\alpha}  +\int_{\partial \Omega} b(u) \varphi \textnormal{d}\vartheta _{\beta} .
\end{aligned}
$$
Additionally, by \ref{43}:
$$
\begin{aligned}
&\left|\int _{\Omega} \langle A( \nabla u) , \varphi \xi _r ^\prime \nabla |x^\prime|\rangle \textnormal{d}\vartheta _{\alpha} \right| \leq \frac{ C_4}{r} \int _{\Omega   \cap\{r \leq |x^\prime| \leq 2 r\}} |\nabla u|^{p-1}  \textnormal{d}\vartheta _{\alpha}\\
&\leq \frac{C_4}{r}\int _{\Omega   \cap\{r \leq |x^\prime| \leq 2 r\}} (|u_1| + |\nabla u_1|)^{p-1}(|u_2| + |\nabla u_2|)^{p-1} \textnormal{d}\vartheta _{\alpha} \\
& \leq  \frac{ C_5}{r} \int _{\Omega ^\prime  \cap\{r \leq |x^\prime| \leq 2 r\}} (|u_1| + |\nabla u_1|)^{p-1} \textnormal{d}x^\prime\\
&\leq  \frac{ C_6}{r} \left(\int _{\Omega ^\prime  \cap\{r \leq |x^\prime| \leq 2 r\}} (|u_1| + |\nabla u_1|)^{p} \textnormal{d}x^\prime  \right)^{\frac{p-1}{p}} r^{\frac{N- d}{p}} \\
& \leq C_7 r^{\frac{N- d -p}{p}} \rightarrow 0 \quad \text { as } r \rightarrow 0 ,
\end{aligned}
$$
where $\Omega ^\prime =  \{x^\prime\:|\: x\in \Omega\}$ and  $C_i$, $i=4, 5, 6$, are positive constants independents of $r$. So, we have obtained that equality \eqref{62} is fulfilled for all $\varphi \in W^{1, p }(\Omega ; \vartheta _{\alpha} ) \cap L^{\infty}(\Omega  )$. Therefore, the singular set $\Gamma$ is removable for solutions of \eqref{1}. \end{proof}

\subsection{Proof of Theorem \ref{23}} We will need the next auxiliary result.

\begin{lemma} \label{267}
Suppose  the hypotheses as in  Lemma \ref{219}  and $N-d >\frac{pq}{q-p+1}$. Then
$$
(1+|x^{\prime \prime}|)^{\delta}\max \{u (x), 0\} \in L^{\infty}( \Omega ) .
$$
\end{lemma}

\renewcommand*{\proofname}{Sketch of the proof}

\begin{proof} 

First we prove
\begin{equation} \label{46}
\max \{u, 0\} \in L^{\infty}( \Omega ).
\end{equation}
We proceed by contradiction. For $r \in\left(0, r_0^2\right)$, we denote
$$
\Lambda(r)=\operatorname{ess} \sup \left\{\max \{  u(x), 0\} \:|\: r \leq |x^\prime|   \leq r_0^2, \ x \in \Omega  \right\} .
$$
We have $\lim _{r \rightarrow 0^{+}} \Lambda(r)=\infty$. For sufficiently small values $r$ we define the function $\psi_r: \mathbb{R} \rightarrow \mathbb{R}$ as follows:
$$
\psi_r(t)= \begin{cases}0 & \text { if } t<r, \\ 1 & \text { if } t>\sqrt{r}, \\ \frac{2}{\ln \frac{1}{r}} \ln \frac{t}{r} & \text { if } r \leq t \leq \sqrt{r} .\end{cases}
$$
Choosing $\rho>0$ such that $\Lambda(\rho)>1$, set
$$
\varphi(x)=\left(\ln \ \max \left\{\frac{   u(x)  }{\Lambda(\rho)}, 1\right\}\right) \psi_r^\gamma (|x^\prime|) ,
$$
where $\gamma= \frac{p q }{q  -p+1}$.   Denote $\Omega  _\rho=\{x \in \Omega   \:|\:    u(x )>\Lambda(\rho)\}$. Observe that, by $(A_1)$, 
$$
\begin{aligned}
 & \int_{\Omega _\rho \cap \{r \leq |x^\prime| \leq \sqrt{r}\}}     |x^\prime|^{-\frac{p q     }{q -p   +1}} \textnormal{d}x \\
 &\leq \int_{\Omega  _\rho \cap \{r \leq |x^\prime| \leq \sqrt{r}\} \cap \{C_0 \geq (1+|x^{\prime \prime}|)^{\delta}\}} |x^\prime|^{-\gamma} \textnormal{d}x 
 \leq C \int_r^{\sqrt{r}}    t^{N-d-1-\gamma} \textnormal{d}  t ,
\end{aligned}
$$
where $C>0$ is a constant independent of $r$. Replacing $\varphi$ into \eqref{222}, and following the approach from the proof of Lemma \ref{67}, we derive \eqref{46}.  Thus, we  conclude the proof of the lemma by virtue of  $(A_1)$ and  \eqref{46}. \end{proof}

Similarly to  Lemma \ref{267},  we obtain the following result.

\begin{lemma}  (Without the  boundary condition)  Assume the hypotheses as in  Lemma \ref{269} and $N-d >\frac{pq}{q-p+1}$. Then
$$
(1+|x^{\prime \prime}|)^{\delta}\max \{u (x) , 0\} \in L^{\infty}(\Omega).
$$
\end{lemma}

We are now in a position to prove the theorem.

\renewcommand*{\proofname}{Proof of Theorem \ref{23}:}
\begin{proof}
First note that $u\in L^{\infty} (\Omega)$, because Lemma \ref{267}. Additionally,  as in \eqref{28}, for small values $r>0$ we have
\begin{equation*}\label{29}
\int_{\Omega   \cap \{r \leq |x^\prime| \leq 2 r \}}|\nabla u|^{p } \textnormal{d}x  \leq C_1 r^{N-d-p}
\end{equation*}
where $C_1$ is a positive constant independent of $r$. Therefore,
$$
\int _{\Omega   \cap\{|x^\prime| \leq 2 r\}}|\nabla u|^{p} \textnormal{d}x   <\infty .
$$
So $|\nabla u| \in L^{p }(\Omega)$,  and thus we have proved that $u \in W^{1, p }(\Omega ) \cap L^{\infty}(\Omega  )$. To finalize the proof, observe that 
$$
\begin{aligned}
&\frac{1}{r}\int _{\Omega \cap \{r\leq |x^\prime|\leq 2r\}} |\nabla u|^{p-1} |\varphi| \textnormal{d}x  \leq \frac{1}{r} \left( \int _{\Omega \cap \{r\leq |x^\prime|\leq 2r\}} |\nabla u|^{p}\textnormal{d}x \right)^{\frac{p-1}{p}} \left( \int _{\Omega} |\varphi|^p \textnormal{d}x \right)^{\frac{1}{p}}\\
&\leq C_1^{\frac{p-1}{p}} r^{\frac{(p-1)(N-d-p) -p}{p}} \left( \int _{\Omega} |\varphi|^p \textnormal{d}x \right)^{\frac{1}{p}},
\end{aligned}
$$
where $\varphi \in W^{1, p }(\Omega ) \cap L^{\infty}(\Omega  )$. Using the reasoning from the proof of Theorem \ref{77}, we finalize the proof. \end{proof}

\noindent\textbf{Author contributions:} All authors have contributed equally to this work for writing, review and editing. All authors have read and agreed to the published version of the manuscript.

\medskip

\noindent\textbf{Funding:} This work were supported by National Institute of Science and Technology of Mathematics INCT-Mat, CNPq, Grants  160951/2022-4, 309998/2020-4, and by Paraíba State Research Foundation (FAPESQ), Grant 3034/2021.

\medskip

\noindent\textbf{Data Availibility:} No data was used for the research described in the article.

\medskip

\noindent\textbf{Declarations}

\medskip

\noindent\textbf{Conflict of interest:} The authors declare no conflict of interest.

\end{document}